\documentclass[11pt]{amsart}
\usepackage{amsmath,amsfonts,amssymb,amsthm,amscd}
\usepackage{graphicx}
\usepackage{mathrsfs}
\usepackage{upgreek}
\usepackage[matrix,arrow,curve]{xy}
\usepackage[T2A]{fontenc}
\usepackage[utf8]{inputenc}
\usepackage{bm}
\usepackage{mathrsfs}

\newcommand\hide[1]{\commented{gray}{Hidden:}{#1}}
\renewcommand\hide[1]\empty

\usepackage{xcolor}

\newtheorem{thm}{Theorem}
\newtheorem{lm}{Lemma}
\newtheorem{cor}{Corollary}

\theoremstyle{definition}

\newtheorem{remark}{Remark}

\theoremstyle{remark}

\begin{document}

\renewcommand{\bm}{\boldsymbol}

\newcommand{\dom}{ {\mathop{\mathrm {dom\,}}\nolimits} }
\newcommand{\ran}{ {\mathop{\mathrm{ran\,}}\nolimits} }

\newcommand{\app}{ {\mathop{\mathrm {app}}\nolimits} }
\newcommand{\ext}{ {\mathop{\mathrm {ext}}\nolimits} }
\newcommand{\cur}{ {\mathop{\mathrm {cur}}\nolimits} }

\newcommand{\rel}{ {\mathop{\mathrm {rel}}\nolimits} }
\newcommand{\fnc}{ {\mathop{\mathrm {fnc}}\nolimits} }

\newcommand{\isom}{\cong}

\newcommand{\inter}{ {\mathop{\mathrm {int}}\nolimits} }
\newcommand{\cl}{ {\mathop{\mathrm {cl}}\nolimits} }
\newcommand{\lcl}{ {\mathop{\mathrm {lcl\,}}\nolimits} }
\newcommand{\rcl}{ {\mathop{\mathrm {rcl\,}}\nolimits} }
\newcommand{\cof}{ {\mathop{\mathrm {cof\,}}\nolimits} }
\newcommand{\add}{ {\mathop{\mathrm {add\,}}\nolimits} }
\newcommand{\sat}{ {\mathop{\mathrm {sat\,}}\nolimits} }
\newcommand{\tc}{ {\mathop{\mathrm {tc\,}}\nolimits} }
\newcommand{\unif}{ {\mathop{\mathrm {unif\,}}\nolimits} }
\newcommand{\uhr}{\!\upharpoonright\!}
\newcommand{\lra}{ {\leftrightarrow} }
\newcommand{\ot}{ {\mathop{\mathrm {ot\,}}\nolimits} }
\newcommand{\ol}{\overline}
\newcommand{\cnc}{ {^\frown} }
\newcommand{\image}{\/``\,}

\newcommand{\scc}{\bm\upbeta}

\newcommand{\wh}{\widehat}
\newcommand{\wt}{\widetilde}
\newcommand{\id}{\mathrm{id}}
\newcommand{\pr}{\mathrm{pr}}
\newcommand{\Cl}{\mathrm{Cl}}
\newcommand{\Clop}{\mathrm{Clop}}
\newcommand{\Lim}{\mathrm{Lim}}
\newcommand{\Lo}{\mathrm{Lo}}
\newcommand{\wf}{\mathrm{wf}}
\newcommand{\WF}{\mathrm{WF}}

\newcommand{\ZF}{\mathrm{ZF}}
\newcommand{\ZFC}{\mathrm{ZFC}}

\newcommand{\RK}{\mathrm{RK}}
\newcommand{\RB}{\mathrm{RB}}
\newcommand{\RF}{\mathrm{RF}}
\newcommand{\FB}{\mathrm{FB}}
\newcommand{\RFB}{\mathrm{RFB}}
\newcommand{\CF}{\mathrm{CF}}
\newcommand{\Comf}{\mathrm{C}}
\newcommand{\K}{\mathrm{K}}
\newcommand{\T}{\mathrm{T}}

\newcommand{\NCF}{\mathrm{NCF}}
\newcommand{\AC}{\mathrm{AC}}
\newcommand{\CH}{\mathrm{CH}}
\newcommand{\GCH}{\mathrm{GCH}}
\newcommand{\SCH}{\mathrm{SCH}}
\newcommand{\FRH}{\mathrm{FRH}}
\newcommand{\MA}{\mathrm{MA}}
\newcommand{\MM}{\mathrm{MM}}
\newcommand{\PFA}{\mathrm{PF}}

\author{Nikolai L.~Poliakov, Denis I.~Saveliev}
\title{On embedding of partially ordered sets 
in $(\scc\omega,\le_\RK)$}

\date{05.05.2025}
\thanks{{\it MSC~2020}:
Primary 
54D35, 
54D80, 
54F05, 
Secondary 
54C08, 
54C20, 
54C25, 
54G10, 
54E99} 

\thanks{The research is supported 
by the MSHE RF GZ project}

\begin{abstract}  
A~natural question, which appeared as Problem~61 
of~\cite{Hart vanMill 2024}, asks whether every finite 
partial order is embeddable in the Rudin--Keisler order 
on the set of ultrafilters over a~countable set. Although 
the positive answer, even for all countable partial 
orders, was proved under~$\CH$ in~\cite{Blass 1970}, 
the question in $\ZFC$ alone remained completely open. 
We show that, in~$\ZFC$, it is possible not only 
to answer in the positive for all countable orders, 
but, moreover, to construct embeddings of 
the ordered by inclusion lattices of finite 
subsets of a~set of cardinality~$2^{\mathfrak c}$, and 
of countable subsets of a~set of cardinality~$\aleph_1$, 
to the set of ultrafilters with any relation lying 
between the Rudin--Keisler and Comfort orders.  
\end{abstract}

\maketitle

\section{Introduction}\label{sec: intro}

\par
The question of the structure of the Rudin--Keisler 
preorder~$\le_{\RK}$  on $\scc\omega$, the set of 
ultrafilters over~$\omega$, and, in particular, of 
what orders can be isomorphically embedded in it, is 
nearly as old as its definition itself. Let us recall, 
without claiming to be exhaustive, several relevant 
facts. Many of them remain true, sometimes with minor 
modifications, in the case of ultrafilters over other 
cardinals; for simplicity, however, we state them here 
only in the case of~$\omega$. 

\par
The preorder~$\le_{\RK}$ was defined independently 
by M.\,E.~Rudin \cite{MERudin 1966,MERudin 1971}
and Keisler \cite{Keisler 1967}; before this, W.~Rudin 
defined the isomorphism of ultrafilters~$\isom$ 
in~\cite{WRudin 1956}. The paper \cite{MERudin 1971} 
contains, among others, the following basic observations 
on the structure of $\le_{\RK}$\,. The equivalence 
induced by $\le_{\RK}$ coincides with~$\isom$\,, 
so $\le_{\RK}$ induces an order on isomorphism 
classes (usually called types) of ultrafilters. 
The preorder $\le_{\RK}$ includes $\le_{\RF}$\,, 
the Rudin--Frol{\'\i}k preorder, appeared 
independently in the works by M.\,E.~Rudin 
\cite{MERudin 1966,MERudin 1971} and Frol{\'\i}k 
\cite{Frolik 1967a,Frolik 1967b,Frolik 1967c}.
For any nonprincipal $\mathfrak u\in\scc\omega$ 
there are $\mathfrak c$ its $\le_{\RK}$-predecessors 
and $2^\mathfrak c$ its $\le_{\RK}$-successors (where 
$\mathfrak c$ is $2^{\aleph_0}$, the cardinality 
of continuum). Moreover, $\le_{\RK}$ is directed, 
and in fact, $\mathfrak c^+$-directed, as follows 
from more general results by Comfort and 
Negrepontis \cite{Comfort Negrepontis 1972}
(see also \cite{Comfort Negrepontis 1974}, 
Theorem~10.9 and onward). It easily follows that 
$(\scc\omega,\le_{\RK})$ contains well-ordered chains 
of (the largest possible) length~$\mathfrak c^+$. 
In~\cite{Kunen 1972}, Kunen shows that there are 
$\le_{\RK}$-incomparable ultrafilters over~$\omega$, 
and in~\cite{Shelah Rudin 1978}, Shelah and 
M.\,E.~Rudin proved that there are $2^\mathfrak c$ 
such ultrafilters. Kunen defined weak $P$-points and 
$\kappa$-OK-points (a~concept weaker than $P$-point but 
stronger than weak $P$-points for $\kappa\ge\aleph_1$) 
and proves in~\cite{Kunen 1978} their existence 
(without any additional assumptions) in the space 
$\scc\omega\setminus\omega$ of nonprincipal ultrafilters 
over~$\omega$. Combining these results, Simon showed 
in~\cite{Simon 1985} that there exist $2^\mathfrak c$ 
pairwise $\le_{\RK}$-incomparable $\aleph_1$-OK 
ultrafilters (thus weak $P$-points), 
the fact we shall use below. 
In~\cite{vanMill 1984}, van~Mill showed that 
$\le_{\RK}$-above each ultrafilter there is 
a~$\mathfrak c$-OK ultrafilter (actually, 
Theorem~4.5.1 of \cite{vanMill 1984} is much 
stronger: it states this for $\le_{\RB}$\,, 
the Rudin--Blass preorder, and that this can be 
realized uniformly, by a~single finite-to-one map).
Product of ultrafilters was defined independently 
by Kochen in \cite{Kochen 1961} and Frayne, Morel, 
and Scott in \cite{Frayne Morel Scott 1962-3} and 
then studied by Frol{\'\i}k~\cite{Frolik 1967a},
Booth~\cite{Booth 1969,Booth 1970}, and others. 
Laflamme in \cite{Laflamme 1989} proved that this 
operation with any nonprincipal fixed argument is 
an isomorphic embedding of $(\scc\omega,<_{\RK})$ 
into itself 
(see also \cite{Garcia-Ferreira 1993}, Theorem~2.26); 
this shows that $<_{\RK}$-above each ultrafilter 
there is an ordered copy of the whole structure 
$(\scc\omega,<_{\RK})$. 
In~\cite{Booth 1969,Booth 1970}, Booth showed that 
the ultrapower of $(\omega,<)$ by any nonprincipal 
ultrafilter embeds in $(\scc\omega,<_{\RF})$; a~fortiori, 
it embeds in $(\scc\omega,<_{\RK})$. 
It follows by a~standard 
saturation argument therefore that both structures 
are universal for linearly ordered sets of
cardinality~$\le\aleph_1$\,.
The Comfort preorder~$\le_{\Comf}$ was 
defined by Comfort (based on the concept 
of $\mathfrak u$-compactness for 
ultrafilters~$\mathfrak u$, introduced 
by Bernstein in~\cite{Bernstein 1970}) and 
first appeared, probably, in Garc{\'\i}a-Ferreira's 
papers~\cite{Garcia-Ferreira 1993, Garcia-Ferreira 1994}. 
Little is known about the structure of~$\le_{\Comf}$ 
on~$\scc\omega$ except that it includes~$\le_{\RK}$\,, 
is $\mathfrak c^+$-directed, and there are 
$2^{\mathfrak c}$ pairwise $\le_{\Comf}$-incomparable 
ultrafilters; the latter is immediate from Simon's 
result mentioned above and the fact that $\le_{\Comf}$ 
coincides with $\le_{\RK}$ on the set of weak $P$-points
(\cite{Garcia-Ferreira 1993}, Theorem~2.10 and 
Corollary~2.11; this last property is unlikely to 
characterize weak $P$-points but this is unknown). 
Using this, we shall embed all finite orders (and more) 
in any relation between $\le_{\RK}$ and~$\le_{\Comf}$\,, 
including the relations~$R_\alpha$ defined by 
the authors in~\cite{Poliakov Saveliev 2025}.
The characterization of $\le_\RK$ via elementary 
embeddings of ultrapowers was given by Blass in his 
pioneering work~\cite{Blass 1970} and was generalized 
to all~$R_\alpha$ (including~$\le_{\Comf}$)
in~\cite{Poliakov Saveliev 2025}. 

\par 
The facts listed above are established in $\ZFC$ 
alone. More statements can be proved under 
additional assumptions, primarily using $\CH$ 
(the continuum hypothesis) or some weaker hypotheses.
In~\cite{WRudin 1956}, W.~Rudin proved that, 
under~$\CH$, the space $\scc\omega\setminus\omega$ 
contains $2^{\mathfrak c}$~$P$-points (and each of them 
can be mapped to each other by some autohomeomorphism 
of the space). In~\cite{MERudin 1971}, M.\,E.~Rudin 
proved, also under~$\CH$, that there is no 
$\le_{\RK}$-maximal $P$-points. 
Pitt in \cite{Pitt 1971} and Mathias in 
\cite{Mathias 1972} independently proved that, 
under~$\CH,$ there is an ultrafilter without 
$P$-points $\le_{\RK}$-below it. 
The existence of $P$-point, unlike weak 
$P$-points, is unprovable in $\ZFC$; this was shown 
by Shelah, see~\cite{Shelah forcing}. 
%
Choquet in~\cite{Choquet 1968a,Choquet 1968b} and 
Sirota in~\cite{Sirota 1969} independently showed 
(in different terms) that $\CH$ implies the existence 
of $\le_{\RK}$-minimal ultrafilters. Any such 
ultrafilter is a~$P$-point (but not conversely), 
so their existence is unprovable in $\ZFC$. In fact, 
$\CH$ implies that the set of $\le_{\RK}$-minimal  
ultrafilters has cardinality~$2^{\mathfrak c}$, and 
Shelah constructed models in which its cardinality 
is any prescribed $\kappa\le\mathfrak c$\,, 
see~\cite{Shelah forcing}. 
There is a~number of natural characterizations of 
$\le_{\RK}$-minimal ultrafilters as Ramsey, selective, 
quasinormal, etc., which were obtained by many authors 
including Blass, Booth, Kunen, Rowbottom, and others 
(see \cite{Comfort Negrepontis 1974}, Theorem~9.6 
and Notes for~\S\,9); also, by the mentioned result 
of~\cite{Garcia-Ferreira 1993}, $\le_{\RK}$-minimal 
ultrafilters are exactly $\le_{\Comf}$-minimal 
weak $P$-points. 
The following results were obtained by Blass, also  
in~\cite{Blass 1970}, under~$\CH$ (or sometimes 
a~weaker assumption). The partially ordered set 
$(\scc\omega,<_\RK)$ contains an isomorphic copy of 
the Boolean lattice $(\mathscr P(\omega),\subset)$ as 
well as of the ``long line'' (the lexicographic product 
of $\omega_1$ and the real line). Each ultrafilter 
has $2^{\mathfrak c}$~immediate $<_\RK$-successors. 
$(\scc\omega,<_\RK)$ is not an upper semilattice, 
moreover, for any~$n<\omega$, there are 
$\mathfrak u,\mathfrak v\in\scc\omega$ with 
exactly $n+2$~minimal upper bounds; furthermore, 
the set of $P$-points is not $<_\RK$-directed (see 
also \cite{Comfort Negrepontis 1974}, Theorem~16.8). 
Ultrafilters with a~prescribed finite number 
of $<_\RK$-predecessors were studied by Blass 
in \cite{Blass 1973} and Daguenet in 
\cite{Daguenet 1975,Daguenet 1979}; 
like  $\le_{\RK}$-minimal ultrafilters, 
they have natural Ramsey-like characteristics. 
%
In contrast to the situation of~$\CH$, under~$\NCF$ 
(the near coherence filter principle, introduced 
by Blass and Shelah in~\cite{Blass Shelah 1987}), 
$(\scc\omega\setminus\omega,<_\RK)$ 
is downward directed without a~least element 
(so there is no $\le_\RK$-minimal ultrafilters) 
and $P$-points form a~coinitial subset. 
The consistency of the existence of simple 
$P_{\aleph_1}$-points and simple $P_{\aleph_2}$-points 
simultaneously, claimed in~\cite{Blass Shelah 1987}, 
was correctly proved recently by Bräuninger and 
Mildenberger in~\cite{Brauninger Mildenberger 2023}. 
In~\cite{Blass 1981}, Blass showed that, 
under~$\CH$, there is an initial segment of 
$(\scc\omega\setminus\omega,<_\RK)$ order-isomorphic
to the tree $((2^{\mathfrak c})^{<\omega_1},\subset)$ 
and such that each increasing $\omega$-sequence of 
its points has a~unique supremum. 
The equality $2^{\mathfrak c}=\mathfrak c^+$ is 
equivalent to the existence of a~well-ordered cofinal 
set of $(\scc\omega,<_\RK)$, as well as of such a~set 
of cardinality~$2^{\mathfrak c}$, or just of a~chain 
of cardinality~$2^{\mathfrak c}$\,; 
this was stated by Comfort and Negrepontis 
in \cite{Comfort Negrepontis 1972} and 
Shelah (see \cite{Comfort Negrepontis 1974}, 
Theorem~10.11 and Corollary~10.15, or 
\cite{Comfort 1977}, Theorem~6.5).
Improving on Blass' result mentioned above, Raghavan 
and Shelah showed in~\cite{Raghavan Shelah 2017} that, 
under $\MA$ (Martin's axiom), $(\scc\omega,<_\RK)$, 
and even its subset consisting of $P$-points, contains 
an isomorphic copy of the Boolean lattice 
$
(\mathscr P(\omega),\subset)
/\mathscr P_\omega(\omega).
$

Returning to the situation in $\ZFC$ alone, it should 
be noted that little is known about the structure of 
the Rudin--Keisler order beyond the facts listed 
above. One natural question that has remained open 
so far is whether $\ZFC$ proves that every finite 
partially ordered set is isomorphically embeddable 
in $(\scc\omega,<_\RK)$; this question has appeared, 
although without explicit mention of provability 
in~$\ZFC$, in Hart and van Mill's list of open 
problems on $\scc\omega$ \cite{Hart vanMill 2024} 
as Problem~61. As is noted there, although 
to answer it affirmatively, it would suffice 
to embed any finite Boolean algebra, i.e., 
$(\mathscr P(n),\subseteq)$ for any finite~$n$, 
even for $n=3$ the question seems to remain open  
(for $n=2$ the solution easily follows 
from known facts, see~\cite{Zwaneveld 2021}). 
In this paper, we answer this question affirmatively. 
Actually, we prove in $\ZFC$ two stronger statements 
(Theorems \ref{t: embedding into RK etal} 
and~\ref{t: embedding of countable into RK etal}). 
We show that $(\scc\omega,\le_\RK)$ contains 
isomorphic copies of the lattice 
$(\mathscr P_\omega(2^{\mathfrak c}),\subseteq)$ 
as well as of the lattice 
$(\mathscr P_{\omega_1}(\omega_1),\subseteq)$;  
moreover, this remains true for $(\scc\omega,R)$ 
with any relation~$R$ lying between $\le_{\RK}$ 
and~$\le_{\Comf}$\,, in particular, with all relations 
$R_{1+\alpha}$ from~\cite{Poliakov Saveliev 2025}. 
We also prove some facts about weak $P$-points, which 
are of an auxiliary importance for the proof of our 
main results, but may be of an independent interest
(Theorems \ref{t: weak p-point} 
and~\ref{t: weak p-point below product}).
The results of this paper were announced 
in~\cite{Poliakov Saveliev 2026}.

\section{Definitions}\label{sec: def}

\par
In this section, we recall the precise definitions 
of the concepts mentioned in the introduction. Most 
of them are given only for the case $X=\omega$, which 
we consider in this paper, but they can easily be 
generalized to arbitrary sets~$X$.

\subsection*{Topology}

\par
For $X$ a~set, $\scc X$ denotes the set of 
ultrafilters over~$X$. As usual, we suppose  
$X\subseteq\scc X$ by identifying each $x\in X$
with the principal ultrafilter given by~$x$. 
A~natural topology on $\scc X$ is generated by 
basic open sets 
$\wt A:=\{\mathfrak u\in\scc X:A\in\mathfrak u\},$
for all $A\subseteq X$. 
The resulting space~$\scc X$ is compact, Hausdorff, 
zero-dimensional (since the sets $\wt A$ are in fact 
clopen), moreover, extremally disconnected (the 
closure of any its open set is open), and the largest 
compactification of the discrete space~$X$. 
In a~more general context, for any Tychonoff 
space~$X$, the symbol $\scc X$ denotes its 
{\it largest} (or {\it \v{C}ech--Stone}) 
{\it compactification}, i.e., a~space satisfying 
the following properties: $X$~is dense in~$\scc X$,  
and every continuous map~$h$ of~$X$ into any compact 
Hausdorff space~$Y$ (uniquely) extends to 
a~continuous map~$\widetilde{h}$ of $\scc X$ into~$Y$. 
Such a~space $\scc X$ is unique up to homeomorphism; 
the extension~$\widetilde{h}$ can be explicitly 
defined by letting, for all $\mathfrak x\in\scc X$,  
\begin{gather*}
\widetilde{h}(\mathfrak x):=
\bigcap\{\cl_Yh[A]:
A~\text{is a~neighborhood of the point }\mathfrak x\}
\end{gather*}
(where $h[A]$ is the image of $A$ under~$h$, 
and $\cl_Y$~is the closure in~$Y$); in the case 
when $X$~is discrete and $\mathfrak u$ is 
an ultrafilter over~$X$, we get 
$
\widetilde{h}(\mathfrak u)=
\bigcap_{A\in\mathfrak u}\cl_Yh[A]
$
(see, e.g., 
\cite{Comfort Negrepontis 1974,Hindman Strauss}). 

\subsection*{Preorders}
The {\it Rudin--Keisler preorder}~$\le_{\RK}$ 
on $\scc\omega$ is defined by letting, 
for all $\mathfrak u,\mathfrak v\in\scc\omega$,
\begin{gather*}
\mathfrak u\le_{\RK}\mathfrak v 
\;\text{ iff }\;
\text{there exists $f:\omega\to\omega$ such that }
\widetilde{f}(\mathfrak v)=\mathfrak u\,. 
\end{gather*}
Requiring additionally that $f$ be finite-to-one, we 
obtain the {\it Rudin--Blass  preorder}~$\le_{\RB}$\,,
and that $f$~be, moreover, a~permutation, the 
{\it isomorphism} relation~$\isom$ of ultrafilters. 
The {\it Rudin--Frol{\'\i}k preorder}~$\le_{\RF}$ 
on $\scc\omega$ is defined by letting, 
for all $\mathfrak u,\mathfrak v\in\scc\omega$,
\begin{gather*}
\mathfrak u\le_{\RF}\mathfrak v 
\;\text{ iff }\;
\text{there exists a~one-to-one 
$f:\omega\to\scc\omega$ such that 
$\ran f$ is strongly discrete and }
\widetilde{f}(\mathfrak u)=\mathfrak v\,,
\end{gather*}
where a~subspace is {\it strongly discrete} iff 
it admits a~pairwise disjoint family of 
neighborhoods of its points. 
(Clearly, in the definition of $\le_{\RK}$ we could 
use a~continuous $f:\scc\omega\to\scc\omega$ 
such that $f(\mathfrak v)=\mathfrak u$, and 
in the definition of $\le_{\RF}$\,, a~topological 
embedding $f:\scc\omega\to\scc\omega$ such that 
$f(\mathfrak u)=\mathfrak v$.)
We have ${\le_\RB}\cup{\le_\RF}\subseteq{\le_\RK}$\,, 
and $\le_{\RF}$ is {\it tree-like}, i.e., for any 
$\mathfrak u$, if $\mathfrak v\le_{\RF}\mathfrak u$ 
and $\mathfrak w\le_{\RF}\mathfrak u$, then 
$\mathfrak v$ and $\mathfrak w$ are 
$\le_{\RF}$-comparable.
The {\em Comfort} preorder~$\le_\Comf$ 
on $\scc\omega$ is defined by letting, 
for all $\mathfrak u,\mathfrak v\in\scc\omega$, 
\begin{gather*}
\mathfrak u\le_\Comf\mathfrak v
\;\text{ iff }\; 
\text{any $\mathfrak v$-compact space 
is $\mathfrak u$-compact,} 
\end{gather*}
where a~space~$X$ is {\em $\mathfrak u$-compact} 
iff $\wt{f}(\mathfrak u)\in X$ for any 
$f:\omega\to X$ (provided $X$ is Tychonoff 
and so included into its largest 
compactification~$\scc X$).
Following~\cite{Poliakov Saveliev 2025}, 
for all ordinals~$\alpha$, define 
the relations~$R_\alpha$ as follows:
$R_0:=\omega\times\scc\omega$,
and for all $\alpha>0$ and 
$\mathfrak u,\mathfrak v\in\scc\omega$, 
\begin{align*}
\mathfrak u\,R_\alpha\,\mathfrak v
\;\text{ iff }\; 
\text{there exists }
f:\omega\to\scc\omega 
\text{ such that }
\widetilde{f}(\mathfrak v)=\mathfrak u 
\text{ and }
f(n)\,R_{<\alpha}\,\mathfrak v, 
\text{ for all }n<\omega,
\end{align*}
where 
$R_{<\alpha}:=\bigcup_{\beta<\alpha}R_\beta$\,. 
The hierarchy of $R_\alpha$ is non-degenerate 
below $\omega_1$ and lies (for $\alpha>0$) 
between $\le_\RK$ and~$\le_\Comf$\,: 
\begin{gather*}
R_0\subset 
{\le_{\RK}}=R_1\subset 
R_2\subset 
\ldots\subset 
R_{<\omega}\subset 
R_\omega\subset 
\ldots\subset
R_{<\alpha}\subset 
R_\alpha\subset 
\ldots\subset 
R_{<\omega_1}=R_{\omega_1}={\le_\Comf}\,.
\end{gather*}
Moreover, if $2\le\alpha\le\omega_1$\,, then 
$R_{<\alpha}$ is a~preorder iff $\alpha$~is 
multiplicatively indecomposable. This allows us 
to define a~hierarchy of preorders~$\le_{\alpha}$
lying between $\le_\RK$ and $\le_\Comf$ 
by letting, for all $\alpha\le\omega_1$\,, 
\begin{gather*}
{\le_0}:={\le_\RK}
\;\text{ and }\; 
{\le_{1+\alpha}}:=R_{<\omega^{\omega^\alpha}}\,. 
\end{gather*}
We speak of the {\it Rudin--Keisler order} both in 
the case of the strict order $<_{\RK}$ on $\scc\omega$ 
and in the case of the (strict or non-strict) order 
on $\scc\omega/{\isom}$ induced by the preorder 
$\le_{\RK}$ on $\scc\omega$\,, and use the same 
symbols in both cases. This convention is adopted 
to other preorders on ultrafilters as well; e.g., 
the {\it Comfort order} may relate to $<_\Comf$ on 
$\scc\omega$ as well as on $\scc\omega/{=_\Comf}$\,.
The question of the existence of an isomorphic embedding 
of a~given order into a~given preorder on ultrafilters, 
which is the topic of this article, can be understood 
equivalently as the embedding between the corresponding 
strict orders and as the embedding into its quotient.


\subsection*{Ultrafilter extensions}

\par
If $\mathfrak u$ is an ultrafilter, define
the (self-dual) {\it ultrafilter quantifier} 
by letting $\forall^{\mathfrak u}x\:\varphi(x)$ 
iff $\{x:\varphi(x)\}\in\mathfrak u$. 
The ({\it Fubini} or {\it tensor}) {\it product of 
ultrafilters} $\mathfrak u,\mathfrak v\in\scc\omega$ 
is an ultrafilter
$
\mathfrak u\otimes\mathfrak v
\in\scc(\omega\times\omega)
$ 
consisting of $S\subseteq\omega\times\omega$ 
such that 
$
\forall^{\mathfrak u}x\,\forall^{\mathfrak v}y\:
(x,y)\in S.
$
More generally, for any $n$-ary map $f:\omega^n\to Y$, 
where $Y$ is a~discrete space,
its {\it ultrafilter extension} is the $n$-ary map 
$\widetilde{f}:(\scc\omega)^n\to\scc Y$ defined by 
letting, for all 
$\mathfrak u_0\,,\ldots,\mathfrak u_{n-1}\in\scc\omega$,
$$
\widetilde{f}(\mathfrak u_0\,,\ldots,\mathfrak u_{n-1}):=
\{S\subseteq Y:
\forall^{\mathfrak u_0}x_0
\ldots
\forall^{\mathfrak u_{n-1}}x_{n-1}\:
f(x_0\,,\ldots,x_{n-1})\in S
\}.
$$
(Thus the multiplication of ultrafilters is 
the ultrafilter extension of the pairing map.)
The ultrafilter extension~$\widetilde{f}$ of $f$ 
is its unique  extension to~$\scc\omega$ that is 
{\it right-continuous with respect to}~$\omega$, 
i.e., such that, for each $i<n$, the shift 
$
\mathfrak v\mapsto
\widetilde{f}(k_0\,,\ldots,k_{i-1},\mathfrak v,
\mathfrak u_{i+1}\,,\ldots,\mathfrak u_{n-1})
$
is continuous whenever 
$k_0,\ldots,k_{i-1}\in\omega$ and 
$
\mathfrak u_{i+1},\ldots,\mathfrak u_{n-1}
\in\scc\omega.
$ 
Let us point out that ultrafilter extensions of 
$n$-ary relations, together with their similar 
topological characterization, are also known; 
moreover, the resulting ultrafilter extensions 
of models are their largest compactifications, 
in a~certain sense. However, since these concepts 
will not be used in this article, we omit here 
their definitions and refer the reader to 
\cite{Saveliev 2011, Saveliev 2012} 
(see also~\cite{Poliakov Saveliev 2021}, Introduction).
Concerning an interplay between ultrafilter extensions 
of maps and the defined relations on ultrafilters, 
we always have: 
$\mathfrak u<_{\RK}\mathfrak u\otimes\mathfrak v$ 
and 
$\mathfrak v<_{\RK}\mathfrak u\otimes\mathfrak v$ 
(via the continuous extensions of the projection maps);   
also 
$\mathfrak u<_{\RF}\mathfrak u\otimes\mathfrak v$ 
and 
$\mathfrak v<_{\RB}\mathfrak u\otimes\mathfrak v\,$, 
but generally neither 
$\mathfrak v\le_{\RF}\mathfrak u\otimes\mathfrak v$ 
nor 
$\mathfrak u\le_{\RB}\mathfrak u\otimes\mathfrak v$ 
(see Remark~\ref{r: non-RB} below). 
Also $\mathfrak u\,R_n\,\mathfrak v$ iff 
there exists $f:\omega^n\to\omega$ such that 
$
\wt{f}(\mathfrak v,\ldots,\mathfrak v)
=\mathfrak u\,,
$
and a~similar characterization can be provided 
for $R_\alpha$ with $\omega\le\alpha<\omega_1$ 
by using certain maps $f:\omega^\omega\to\omega$ 
(see~\cite{Poliakov Saveliev 2025}).

\subsection*{$P$-point variants}

\par 
The following concepts are meaningful for 
any topological space, but below we shall 
always assume that they refer to the space 
$\scc\omega\setminus\omega$ 
of non-principal ultra\-filters (which is 
closed in~$\scc\omega$). An ultrafilter 
$\mathfrak u\in\scc\omega\setminus\omega$ 
is a~{\it $P$-point} iff the intersection of any 
countable set of neighborhoods of~$\mathfrak u$ 
contains a~neighborhood of~$\mathfrak u$, and 
a~{\it weak $P$-point} iff $\mathfrak u$ does 
not belong to the closure of any countable set 
$
A\subseteq\scc\omega\setminus
(\omega\cup\{\mathfrak u\}).
$
More generally, given a~cardinal~$\kappa$, an 
ultrafilter $\mathfrak u\in\scc\omega\setminus\omega$ 
is a~{\it $P_\kappa$-point} iff the intersection 
of~$<\kappa$ neighborhoods of~$\mathfrak u$ 
contains a~neighborhood of~$\mathfrak u$, and 
a~{\it weak $P_\kappa$-point} 
iff it does not belongs to the closure of any set 
$
A\subseteq\scc\omega\setminus
(\omega\cup\{\mathfrak u\})
$
of cardinality $|A|<\kappa$. 
(Thus $P$-points and weak $P$-points 
are obtained when $\kappa=\aleph_1$\,.)
A~$P_\kappa$-point is {\it simple} iff 
it has a~basis of cardinality~$\kappa$. 
An ultrafilter~$\mathfrak u$ is $\kappa$-{\it OK} 
iff for any $\omega$-sequence $(A_n)_{n<\omega}$ 
of neighborhoods of~$\mathfrak u$ there exists 
its $\kappa$-{\it refinement}, i.e., 
a~$\kappa$-sequence $(B_\alpha)_{\alpha<\kappa}$ 
of neighborhoods of~$\mathfrak u$ such that, 
for any $n<\omega$ and $n$-subsequence 
$(B_{\alpha_i})_{i<n}$\,, we have 
$\bigcap_{i<n}B_{\alpha_i}\subseteq A_n$\,. Obviously, 
each of the three properties depending on~$\kappa$ 
(i.e., of being $P_\kappa$\,, weak $P_\kappa$\,, and 
$\kappa$-OK) become stronger as $\kappa$ increases; 
any $\kappa$-OK point is a~weak $P_\kappa$-point; 
and any $P$-point is $\kappa$-OK for all~$\kappa$. 
Actually, we shall only need weak $P$-points for 
our proofs.

\section{Characterizations of weak $P$-points}%
\label{sec: weak p-points}

\par 
In this section, we provide some characterizations of 
weak $P$-points, which will be used in the subsequent 
proof of our main results, and are also of some 
independent interest. 

\par 
Clearly, a~unary operation on $\scc\omega$ is 
continuous iff it is a~continuous extension of 
a~map on~$\omega$ into~$\scc\omega$. Let us provide 
an alternative criterion in terms of binary operations 
on~$\omega$, which is easy but, to our knowledge, not 
previously published in the literature.

\begin{lm}\label{l: cont}
A~map $f:\scc\omega\to\scc\omega$ is 
continuous iff there exist $g:\omega^2\to\omega$ 
and $\mathfrak v\in\scc\omega$ such that 
$$
f(\mathfrak u)=\widetilde g(\mathfrak u,\mathfrak v)
$$
for all $\mathfrak u\in\scc\omega$. 
Moreover, as such a~$\mathfrak v$ one can take 
any $\le_{\RK}$-upper bound of $f[\omega]$, 
the image of $\omega$ under~$f$.
\end{lm}

\begin{proof}
Since the extension 
$\widetilde{g}:(\scc\omega)^2\to\scc\omega$ 
is right-continuous, 
the map~$f$ obtained from $\wt{g}$ by letting 
$
f(\mathfrak u):=
\widetilde g(\mathfrak u,\mathfrak v),
$ 
is automatically continuous. It remains to prove 
the converse.

\par 
Let $f:\scc\omega\to\scc\omega$ be continuos. 
Using that any countable (and in fact, of 
cardinality~$\mathfrak c$) set of ultrafilters 
there exists an $\le_{\RK}$-upper bound, 
pick $\mathfrak v\in\scc\omega$ such that 
$f(i)\le_{\RK}\mathfrak v$, for all $i\in\omega$. 
For each $i\in\omega$ we fix some 
$e_i:\omega\to\omega$ such that 
$\widetilde{e_i}(\mathfrak v)=\mathfrak u_i$\,. 
Consider the map $g:\omega^2\to\omega$ defined by 
letting $g(i,j):=e_i(j)$, and furthermore, its 
extension $\widetilde g:(\scc\omega)^2\to\scc\omega$. 
For every fixed $i\in\omega$ we have 
$$
g_{(i)}(j)=g(i,j)=e_i(j)
$$
(where the map $g_{(i)}$ is obtained from $g$ 
by fixing $i$ as the first argument), hence,
$$
\widetilde g(i,\mathfrak v)=
\widetilde{g_{(i)}}(\mathfrak v)=
\widetilde{e_i}(\mathfrak v)=
\mathfrak u_i=f(i).
$$
Now consider $h:\omega\to\scc\omega$ defined 
by letting $h(i):=\widetilde{g}(i,\mathfrak v)$. 
The map~$h$ coincides with the restriction of $f$ 
to~$\omega$. Therefore,  
$$
\widetilde{g}(\mathfrak u,\mathfrak v)
=\widetilde{h}(\mathfrak u)
=f(\mathfrak u)
$$
for all $\mathfrak u\in\scc\omega$, as required.   
\end{proof}

\par 
The following fact is well known:

\begin{lm}\label{l: closure}
For every $\mathfrak{u}\in\scc\omega$ and 
$A\subseteq\scc\omega$, the following are 
equivalent: 
\begin{enumerate}
\item[(i)]
$\mathfrak u\subseteq\bigcup A\,;$
\item[(ii)] 
$\bigcap A\subseteq\mathfrak u\,;$
\item[(iii)]
$\mathfrak u\in\cl_{\scc\omega}A\,.$
\end{enumerate}  
\end{lm}

\begin{proof}
This is a~routine check.  

\par 
(i)$\Rightarrow$(ii). 
Let $\mathfrak u\subseteq\bigcup A.$ If 
$S\in\bigcap A$ then $\omega\setminus S$ 
is not in any $\mathfrak v\in A$, hence, 
$\omega\setminus S\notin\mathfrak u$\,, and so, 
$S\in\mathfrak u$\,.

\par 
(ii)$\Rightarrow$(i).
Let $\bigcap A\subseteq\mathfrak u\,.$ Let 
$S\in\mathfrak u$. Assume $S$ is not in 
any $\mathfrak v\in A$. Then 
$\omega\setminus S\in\bigcap A.$ Therefore,  
$\omega\setminus S\in\mathfrak u$\,, 
a~contradiction. 

\par 
It follows from (i)$\Leftrightarrow$(ii) that 
$
\cl_{\scc\omega}A=
\bigl\{\mathfrak u\in\scc\omega:
\bigcap A\subseteq\mathfrak u\bigr\},
$ 
whence we get (ii)$\Leftrightarrow$(iii).
\end{proof}

\begin{lm}\label{l: closure via cont}
For every $\mathfrak{u}\in\scc\omega$ 
and countable $A\subseteq\scc\omega$, 
the following are equivalent: 
\begin{enumerate}
\item[(i)] 
$\mathfrak u\in\cl_{\scc\omega}A\,;$ 
\item[(iia)] 
$\mathfrak u=\widetilde{f}(\mathfrak v)$
for any $f:\omega\to\scc\omega$ 
such that $\ran{f}=A$ and 
some $\mathfrak v\in\scc\omega\,;$ 
\item[(iib)] 
$\mathfrak u=\widetilde{f}(\mathfrak v)$
for some $f:\omega\to\scc\omega$ 
such that $\ran{f}\subseteq A$ and 
some $\mathfrak v\in\scc\omega\,;$ 
\item[(iii)]
$\mathfrak u=\widetilde{g}(\mathfrak v,\mathfrak w)$ 
for some $g:\omega^2\to\omega$ and some 
$\mathfrak v,\mathfrak w\in\scc\omega$ such that 
$
\{\widetilde{g}(i,\mathfrak w):
i\in\omega\}\subseteq A\,.
$  
\end{enumerate}
\end{lm}

\begin{proof}
(i)$\Rightarrow$(iia).
First let us note the following fact. 
Let $A\subseteq\scc\omega$ have an 
arbitrary cardinality $\kappa\le2^{\mathfrak c}$, 
and let $\mathfrak u_\alpha\,,\,{\alpha<\kappa}$, 
be an enumeration of $A$ (possibly with repetitions). 
If $\mathfrak u\in\cl_{\scc\omega}A$, let 
$\mathscr V:=\{A_S:S\in\mathfrak u\}$ where  
$A_S:=\{\alpha<\kappa:S\in\mathfrak u_\alpha\}$. 
As easy to see, for all $S,T\in\mathfrak u$ 
we have: 
\begin{enumerate}
\item[(a)] 
$A_S\ne\emptyset$ (due to 
$\mathfrak u\subseteq\bigcup A$ 
by Lemma~\ref{l: closure}) and 
\item[(b)]
$A_{S}\cap A_{T}=A_{S\cap T}$ 
(since $S\cap T\in\mathfrak u$).
\end{enumerate} 
Thus $\mathscr V$~does not have the empty set 
and is centered. 

\par 
Now let $\kappa:=\omega$. 
If $f:\omega\to\scc\omega$ is such that 
$\mathrm{ran}\,f=A$, letting 
$\mathfrak u_i:=f(i)$, $i<\omega$, 
we get an enumeration of~$A$. Construct 
$\mathscr V$ as above and pick any ultrafilter 
$\mathfrak v$ that extends~$\mathscr V$. Then 
$\mathfrak u=\widetilde{f}(\mathfrak v)$. 
Indeed, let $S\in\mathfrak u$\,. Then 
$A_S\in\mathfrak v$ and for every $i\in A_S$ 
we have $S\in f(i)$, whence it follows 
$\widetilde{f}(\mathfrak v)=\mathfrak u$\,.  

\par 
(iia)$\Rightarrow$(iib). Clear.

\par 
(iib)$\Rightarrow$(i).
If for some $\mathfrak v\in\scc\omega$ and 
$f:\omega\to\scc\omega$ we have 
$\widetilde{f}(\mathfrak v)=\mathfrak u$\,, 
then $\mathfrak u\in\cl_{\scc\omega}\,\ran{f}$ 
(since $\widetilde f(\mathfrak v)$ is the 
intersection of the sets $\cl_{\scc\omega}\,f[S]$ 
for all $S\in\mathfrak v$). 

\par 
(iia)$\Leftrightarrow$(iii). 
Use Lemma~\ref{l: cont}.
\end{proof}

\par 
Clearly, we can reformulate~(iia) by changing 
arbitrary $f:\omega\to\scc\omega$ with $\ran f=A$
to arbitrary continuous $f:\scc\omega\to\scc\omega$ 
with $f[\omega]=A$; and similarly for~(iib). 
A~variant of (iii) of the form ``\,for any 
$g:\omega^2\to\omega$\,'' is also possible 
but rather cumbersome. 

\par 
The above leads to the following characterizations 
of weak $P$-points: 

\begin{thm}\label{t: weak p-point}
For every $\mathfrak u\in\scc\omega\setminus\omega$ 
the following are equivalent:   
\begin{enumerate}
\item[(i)] 
$\mathfrak u$~is a~weak $P$-point$\,;$
\item[(iia)] 
for all $f:\omega\to\scc\omega$ and 
$\mathfrak v\in\scc\omega$ such that 
$\mathfrak u=\wt{f}(\mathfrak v)$ 
there exists $i\in\omega$ satisfying
$f(i)=\mathfrak u$ or $f(i)\in\omega\,;$
\item[(iib)] 
for all $f:\omega\to\scc\omega$ and 
$\mathfrak v\in\scc\omega$ such that 
$\mathfrak u=\wt{f}(\mathfrak v)$ 
one of two sets
$\{i\in\omega:f(i)=\mathfrak u\}$ 
and $\{i\in\omega:f(i)\in\omega\}$ 
belongs to~$\mathfrak v\,;$
\item[(iiia)]
for all $g:\omega^2\to\omega$ and 
$\mathfrak v,\mathfrak w\in\scc\omega$ such that 
$\mathfrak u=\widetilde g(\mathfrak v,\mathfrak w)$ 
there exists $i\in\omega$ satisfying 
$\widetilde{g}(i,\mathfrak w)=\mathfrak u$ or 
$\widetilde{g}(i,\mathfrak w)\in\omega\,;$
\item[(iiib)] 
for all $g:\omega^2\to\omega$ and  
$\mathfrak v,\mathfrak w\in\scc\omega$ such that 
$\mathfrak u=\widetilde g(\mathfrak v,\mathfrak w)$ 
one of two sets  
$\{i\in\omega:g(i,\mathfrak w)=\mathfrak u\}$ 
and $\{i\in\omega:g(i,\mathfrak w)\in\omega\}$ 
belongs to~$\mathfrak v\,.$
\end{enumerate}
\end{thm}

\begin{proof}
(i)$\Leftrightarrow$(iia)$\Leftrightarrow$(iiia).
This follows immediately from 
Lemma~\ref{l: closure via cont}.

\par 
(iia)$\Leftrightarrow$(iib).  
The implication from the right to the left is 
evident, so let us prove the converse. 
Assume that 
neither the set $\{i\in\omega:f(i)=\mathfrak u\}$ 
nor the set $\{i\in\omega:f(i)\in \omega\}$ 
belong to~$\mathfrak v$\,. Then $\mathfrak v$ 
has the set 
$$
A:=
\{i\in\omega:f(i)\notin\omega\cup\{\mathfrak u\}\}.
$$ 
Pick an arbitrary ultrafilter 
$
\mathfrak w\in
\scc\omega\setminus(\omega\cup\{\mathfrak u\})
$ 
and consider the map $e:\omega\to\scc\omega$ 
defined by letting 
$$
e(i):=
\begin{cases}
f(i)&\text{if }i\in A, 
\\
\mathfrak w&\text{otherwise}. 
\end{cases}
$$
Then 
$
\wt{e}(\mathfrak v)=\wt{f}(\mathfrak v)
=\mathfrak u\,,
$ 
but $e(i)\notin\omega\cup\{\mathfrak u\}$ 
for all $i\in\omega$. 

\par 
(iiia)$\Leftrightarrow$(iiib).  
Likewise. 
\end{proof}

\begin{remark}
Moreover, each of (i)--(iiib)
in Theorem~\ref{t: weak p-point} 
is equivalent to any of the two following: 
\begin{enumerate}
\item[(iva)] 
for all $g:\omega^2\to\omega$ and 
$\mathfrak v\in\scc\omega$ such that 
$\mathfrak u=\widetilde g(\mathfrak v,\mathfrak v)$ 
there exists $i\in\omega$ satisfying 
$\widetilde g(i,\mathfrak v)=\mathfrak u$ or 
$\widetilde g(i,\mathfrak v)\in\omega\,;$
\item[(ivb)] 
for all $g:\omega^2\to\omega$ and  
$\mathfrak v\in\scc\omega$ such that 
$\mathfrak u=\widetilde g(\mathfrak v,\mathfrak v)$ 
one of two sets  
$\{i\in\omega:g(i,\mathfrak v)=\mathfrak u\}$ 
and $\{i\in\omega:g(i,\mathfrak v)\in\omega\}$ 
belongs to~$\mathfrak v\,.$
\end{enumerate}
We prove this fact elsewhere; here it will not be used. 
\end{remark}

\begin{remark}\label{r: larger kappa} 
The results of this section can be generalized
using cardinal parameters (but still considering 
ultrafilters in~$\scc\omega$).

\par 
Lemma~\ref{l: cont} is generalized as follows: 
given a~$\kappa\le\mathfrak c$\,, 
a~map $f:\scc\kappa\to\scc\omega$ is continuous 
iff there exist $g:\kappa\times\omega\to\omega$ 
and $\mathfrak v\in\scc\omega$ such that 
$$
f(\mathfrak u)=\widetilde g(\mathfrak u,\mathfrak v)
$$
for all $\mathfrak u\in\scc\kappa$. Indeed, 
since $f[\kappa]$, the image of $\kappa$ under~$f$, 
is of cardinality~$\le\kappa\le\mathfrak c$\,, it 
has an $\le_{\RK}$-upper bound, so we can apply 
the same argument as in the proof of Lemma~\ref{l: cont}.

\par 
Lemma~\ref{l: closure via cont} is also easily 
extended to uncountable sets of ultrafilters: 
if $A\subseteq\scc\omega$ is of 
cardinality~$\kappa$\,($\le2^{\mathfrak c}$),
then the following are equivalent: 
\begin{enumerate}
\item[(i)] 
$\mathfrak u\in\cl_{\scc\omega}\,A\,;$
\item[(iia)] 
$\mathfrak u=\widetilde{f}(\mathfrak v)$
for any $f:\kappa\to\scc\omega$ such that 
$\ran{f}=A$ and some $\mathfrak v\in\scc\kappa\,;$ 
\item[(iib)]
$\mathfrak u=\widetilde f(\mathfrak v)$ 
for some $f:\kappa\to\scc\omega$ 
such that $\ran{f}\subseteq A$ and 
some $\mathfrak v\in\scc\kappa\,;$ 
\end{enumerate} 
and if, moreover, $\kappa\le\mathfrak c$\,, then each 
of {\rm (i)--(iib)} is equivalent to the following: 
\begin{enumerate}
\item[(iii)] 
$\mathfrak u=\widetilde{g}(\mathfrak v,\mathfrak w)$ 
for some $g:\kappa\times\omega\to\omega$ 
and some $\mathfrak v\in\scc\kappa$,
$\mathfrak w\in\scc\omega$  
such that 
$
\{\widetilde{g}(\alpha,\mathfrak w):
\alpha\in\kappa\}\subseteq A\,.
$ 
\end{enumerate}

\par 
Finally, using this, we can generalize 
Theorem~\ref{t: weak p-point} to characterize 
weak $P_{\kappa}$-points:
for every $\kappa<\mathfrak c$ and 
$\mathfrak u\in\scc\omega\setminus\omega$ 
the following are equivalent: 
\begin{enumerate}
\item[(i)] 
$\mathfrak u$~is a~weak $P_{\kappa^+}$-point$\,;$
\item[(iia)] 
for all $f:\kappa\to\scc\omega$ and 
$\mathfrak v\in\scc\kappa$ such that 
$\mathfrak u=\wt{f}(\mathfrak v)$ 
there exists $\alpha\in\kappa$ satisfying 
$f(\alpha)=\mathfrak u$ or $f(\alpha)\in\omega\,;$
\item[(iib)] 
for all $f:\kappa\to\scc\omega$ and 
$\mathfrak v\in\scc\kappa$ such that 
$\mathfrak u=\wt{f}(\mathfrak v)$ 
one of two sets  
$\{\alpha\in\kappa:f(\alpha)=\mathfrak u\}$ 
and $\{\alpha\in\kappa:f(\alpha)\in\omega\}$ 
belongs to~$\mathfrak v\,;$
\item[(iiia)]
for all $g:\kappa\times\omega\to\omega$ and 
$\mathfrak v\in\scc\kappa$, $\mathfrak w\in\scc\omega$ 
such that  
$\mathfrak u=\widetilde g(\mathfrak v,\mathfrak w)$ 
there exists $\alpha\in\kappa$ satisfying 
$\widetilde{g}(\alpha,\mathfrak w)=\mathfrak u$ 
or $\widetilde{g}(\alpha,\mathfrak w)\in\omega\,;$
\item[(iiib)]
for all $g:\kappa\times\omega\to\omega$ and 
$\mathfrak v\in\scc\kappa$, $\mathfrak w\in\scc\omega$ 
such that  
$\mathfrak u=\widetilde g(\mathfrak v,\mathfrak w)$ 
one of two sets  
$\{\alpha\in\kappa:g(\alpha,\mathfrak w)=\mathfrak u\}$ 
and 
$\{\alpha\in\kappa:g(\alpha,\mathfrak w)\in\omega\}$ 
belongs to~$\mathfrak v\,.$
\end{enumerate}

\par 
Since we shall not use these generalizations to obtain 
our main results, we leave it to the reader to write 
out simple modifications of the previous proofs.
\end{remark}

\begin{remark}\label{r: non-RB}
Recall that for all  
$\mathfrak u,\mathfrak v\in\scc\omega$ we have 
$\mathfrak u\le_{\RF}\mathfrak u\otimes\mathfrak v$ 
(since 
$
\widetilde{f}(\mathfrak u)\isom   
\mathfrak u\otimes\mathfrak v
$ 
whenever $\{f(i):i\in\omega\}$ is strongly discrete 
and $f(i)\isom\mathfrak v$ for all $i\in\omega$) and 
$\mathfrak v\le_{\RB}\mathfrak u\otimes\mathfrak v$ 
(since 
$
\widetilde{f}(\mathfrak u\otimes\mathfrak v)
=\mathfrak v
$
where a~finite-to-one $f:\omega^2\to\omega$ is defined 
by letting $f(i,j):=\max(i,j)$  for all $i,j\in\omega$). 
However, two other inequalities generally fail. 
In the case of $\le_{\RF}$, this is easy: 
if $\mathfrak u$ and $\mathfrak v$ 
are $\le_{\RK}$-incomparable, then 
$\mathfrak v\le_{\RF}\mathfrak u\otimes\mathfrak v$
fails (just because 
$\mathfrak u\le_{\RF}\mathfrak u\otimes\mathfrak v$ 
and $\le_{\RF}$ is tree-like). 
The case of $\le_{\RB}$ can be also easily handled 
by using our characterization of weak $P$-points 
(Theorem~\ref{t: weak p-point}); let us show that, 
again if $\mathfrak u$ and $\mathfrak v$ are 
$\le_{\RK}$-incomparable, then 
$\mathfrak u\le_{\RB}\mathfrak u\otimes\mathfrak v$ 
fails whenever $\mathfrak u$ is a~weak $P$-point. 

\par 
Indeed, toward a~contradiction, assume there is 
a~finite-to-one map $g:\omega^2\to\omega$ such that 
$\widetilde{g}(\mathfrak u,\mathfrak v)=\mathfrak u\,$. 
Since $\mathfrak u$ is a~weak $P$-point, by 
Theorem~\ref{t: weak p-point}(iiia), there is 
$i\in\omega$ such that one of the two cases holds: 
either $\widetilde{g}(i,\mathfrak v)\in\omega$,
or $\widetilde g(i,\mathfrak v)=\mathfrak u\,.$ 
Let us consider them separately. 
First define an auxiliary map $h:\omega\to\omega$ 
by letting $h(j):=f(i,j)$ for all $j<\omega$. Thus,  
$\widetilde h(\mathfrak v)=\widetilde g(i,\mathfrak v).$
Now, in the first case, we have 
$\widetilde{h}(\mathfrak v)\in\omega$, 
so $h$~is constant on some $A\in\mathfrak v\,$. 
Therefore, $g$~is constant on $\{i\}\times A$, 
and so $g$~is not finite-to-one, which contradicts to 
the assumption on~$g$.
And in the second case, we have 
$\widetilde h(\mathfrak v)=\mathfrak u\,,$ and so 
$\mathfrak u\le_{\RK}\mathfrak v\,$, which contradicts 
to the assumption that $\mathfrak u$ and~$\mathfrak v$ 
are $\le_{\RK}$-incomparable. 
\end{remark}

\section{Embedding of the lattice 
$(\mathscr P_{\omega}(2^{\mathfrak c}),\subseteq)$}%
\label{sec: first main thm}

\par 
In this section, we prove the first of two our main 
results, Theorem~\ref{t: embedding into RK etal}, 
which isomorphically embeds the lattice of finite 
subsets of a~set of cardinality~$2^{\mathfrak c}$ 
ordered by inclusion in $\scc\omega$ with 
the Rudin--Keisler preorder. Clearly, this solves 
Problem~61 of~\cite{Hart vanMill 2024} in the 
affirmative way. Moreover, our construction gives 
an embedding in the Comfort preorder as well, and even 
in any relation between the two preorders. 

\par 
We start with the following result, which shows that 
weak $P$-points lying $\le_{\Comf}$-below products 
of ultrafilters distribute $\le_{\RK}$-below some 
of their factors; this is vaguely reminiscent of 
Blass' result on ultrafilters below products of 
$\le_{\RK}$-minimal ultrafilters (see \cite{Blass 1970}, 
Theorem~21.1); however, our result is obtained in 
$\ZFC$ alone. 

\begin{thm}\label{t: weak p-point below product}
Let $\mathfrak u\in\scc\omega$ be a~weak $P$-point, 
$\mathfrak v_0\,,\ldots,\mathfrak v_n\in\scc\omega$ 
arbitrary ultrafiltres. If\,
$
\mathfrak u\le_{\Comf} 
\mathfrak v_0\otimes\ldots\otimes\mathfrak v_n
$\,, 
then\, $\mathfrak u\le_\RK\mathfrak v_k$ 
for some $0\le k\le n$.
\end{thm}

\begin{proof}
Let us first recall that, whenever $\mathfrak u$ 
is a~weak $P$-point and $\mathfrak v$ an arbitrary 
ultrafilter, then $\mathfrak u\leq_{\Comf}\mathfrak v$ 
implies $\mathfrak u\leq_{\RK}\mathfrak v$ 
(see~\cite{Garcia-Ferreira 1993}, Theorem~2.10), 
so we can assume 
$
\mathfrak u\le_\RK
\mathfrak v_0\otimes\ldots\otimes\mathfrak v_n\,. 
$ 
Now we prove the statement by induction on~$n$. 

\par 
For $n=0$, the assertion holds trivially. 
Let $n=1$, so
$\mathfrak u\le_\RK\mathfrak v_0\otimes\mathfrak v_1$\,,
and let $g:\omega^2\to\omega$ be such that
$
\mathfrak u=
\widetilde g(\mathfrak v_0\otimes\mathfrak v_1)=
\widetilde g(\mathfrak v_0\,,\mathfrak v_1),
$
where the map $g$ on the left is treated as unary 
and on the right as binary.
Then, by Theorem~\ref{t: weak p-point}, either the set
$
A:=\{i\in\omega:
\widetilde{g}(i,\mathfrak v_1)=\mathfrak u\}
$ 
or the set 
$
B:=\{i\in\omega:
\widetilde{g}(i,\mathfrak v_1)\in\omega\}
$
belongs to~$\mathfrak v_0$\,. Let us consider 
both cases separately. 

\par 
If $A\in\mathfrak v_0$\,, pick any $i\in A$. Then
$
\widetilde{g}(i\otimes\mathfrak v_1)=
\widetilde{g}(i,\mathfrak v_1)=
\mathfrak u
$ 
(where again the map $g$ on the left is treated 
as unary and on the right as binary), and so, 
$\mathfrak u\le_{\RK}i\otimes\mathfrak v_1$\,. 
Since $i\otimes\mathfrak v_1$ is isomorphic 
to~$\mathfrak v_1$\,, we get 
$\mathfrak u\le_\RK\mathfrak v_1$\,. 

\par 
If $B\in\mathfrak v_0$\,, define
$f:\omega\to\omega$ by letting 
$$
f(i):=
\begin{cases}
\widetilde{g}(i,\mathfrak v_1)
&\text{if }i\in B,
\\
0
&\text{otherwise}.
\end{cases}
$$
Since $\widetilde{g}$ is right-continuous, 
the map~$h$ defined by letting 
$
h(\mathfrak v):=
\widetilde{g}(\mathfrak v\,,\mathfrak v_1) 
$
is continuous, hence, $h$~coincides with 
the map~$\widetilde{f}$ on~$\widetilde{B}$  
(since they coincide on~$B$). So we get 
$
\widetilde{f}(\mathfrak v_0)=
h(\mathfrak v_0)=
\widetilde{g}(\mathfrak v_0\,,\mathfrak v_1)=
\mathfrak u\,,
$ 
and, therefore, $\mathfrak u\le_\RK\mathfrak v_0$\,. 

\par 
Thus, for $n=1$, the assertion also holds. 
Further we argue by induction. Assume that 
holds for~$n$. Let
$
\mathfrak v:=
\mathfrak v_0\otimes\ldots\otimes\mathfrak v_n\,,
$ 
and assume 
$
\mathfrak u\le_{\RK}
\mathfrak v_0\otimes\ldots\otimes
\mathfrak v_n\otimes\mathfrak v_{n+1}
\isom\mathfrak v\otimes\mathfrak v_{n+1}
$\,. 
Then either $\mathfrak u\le_\RK\mathfrak v$, 
and so, $\mathfrak u\le_\RK\mathfrak v_k$ for 
some $0\le k\le n$ by the induction hypothesis, 
or $\mathfrak u\le_\RK\mathfrak v_{n+1}$\,. 

\par 
The proof is complete.
\end{proof}

\begin{remark}
Theorem~\ref{t: weak p-point below product} 
admits some generalizations. 

\par 
A~trivial generalization is obtained by 
replacing the ultrafilter multiplication by 
the ultraextension of an arbitrary operation:
let $\mathfrak u\in\scc\omega$ be a~weak $P$-point, 
then for any $e:\omega^{n+1}\to\omega$ and arbitrary 
$\mathfrak v_0\,,\ldots,\mathfrak v_n\in\scc\omega$,  
$
\mathfrak u\le_\RK
\widetilde{e}(\mathfrak v_0\,,\ldots,\mathfrak v_n)
$
implies $\mathfrak u\le_\RK\mathfrak v_k$ 
for some $k\in\omega$
(this is evident from 
$
\widetilde{e}(\mathfrak v_0\,,\ldots,\mathfrak v_n)
\leq_{\RK}
\mathfrak v_0\otimes\ldots\otimes\mathfrak v_n
$). 

\par 
Also Theorem~\ref{t: weak p-point below product} 
can be generalized to weak $P_{\kappa}$-points
as follows. Suppose $\kappa<\mathfrak c$\,, 
$\mathfrak u\in\scc\omega$ is 
a~weak $P_{\kappa^+}$-point,  
and $\mathfrak v_0\in\scc\kappa$ and 
$
\mathfrak v_1\,,\ldots,\mathfrak v_{n}
\in\scc\omega
$ 
arbitrary. If 
$
\mathfrak u\le_\RK
\mathfrak v_0\otimes\mathfrak v_1
\otimes\ldots\otimes\mathfrak v_n
$\,, 
then $\mathfrak u\le_\RK\mathfrak v_k$ 
for some $0\le k\le n$. To prove this, use 
the characterization of weak $P_{\kappa^+}$-points 
mentioned in Remark~\ref{r: larger kappa} 
and modify the proof above. 
\end{remark}

\begin{thm}[The First Main Theorem]%
\label{t: embedding into RK etal}
There exists a~map 
$e:\mathscr P_\omega(\scc\omega)\to\scc\omega$ 
such that for all 
$X,Y\in\mathscr P_\omega(\scc\omega)$, 
\begin{enumerate}
\item[(i)] 
if $X\subseteq Y$ then $e(X)\le_\RK e(Y)\,;$
\item[(ii)] 
if $e(X)\leq_{\Comf}e(Y)$ then $X\subseteq Y$. 
\end{enumerate}
Consequently, $e$~is an isomorphic embedding of 
the lattice $(\mathscr P_\omega(\scc\omega),\subseteq)$ 
into the structure $(\scc\omega,R)$ whenever 
${\le_\RK}\subseteq{R}\subseteq{\le_\Comf}$\,. 
\end{thm}

\begin{proof}
By~\cite{Simon 1985}, there exists 
$A\subseteq\scc\omega$ consisting of pairwise 
$\le_{\RK}$-incomparable weak $P$-points 
such that $|A|=|\scc\omega|$.
Fix such an~$A$ and enumerate its points by 
$\mathfrak u_\alpha$\,, $\alpha<2^{\mathfrak c}$\,. 
We shall construct the required map~$e$ whose 
domain is rather $\mathscr P_\omega(A)$ 
than $\mathscr P_\omega(\scc\omega)$.
For $X\in\mathscr P_\omega(A)$, if 
$
X=\{
\mathfrak u_{\alpha_0},\ldots,\mathfrak u_{\alpha_n}
\}
$
and $\alpha_0<\ldots<\alpha_n<2^{\mathfrak c}$\,, 
let 
$$
e(X):=
\mathfrak u_{\alpha_0}
\otimes
\ldots
\otimes
\mathfrak u_{\alpha_n}\,. 
$$
Let us verify that, for all 
$X,Y\in\mathscr P_{\omega}(A)$, 
both (i) and~(ii) hold. 
Let 
$
X:=
\{\mathfrak u_{\alpha_0},
\ldots,
\mathfrak u_{\alpha_m}
\},
$
$\alpha_0<\ldots<\alpha_m<2^{\mathfrak c}$, and  
$
Y:=
\{
\mathfrak u_{\beta_0},
\ldots, 
\mathfrak u_{\beta_n}
\},
$
$\beta_0<\ldots<\beta_n<2^{\mathfrak c}$.

\par 
(i). 
Clear.  

\par 
(ii).  
Assume $e(X)\leq_{\Comf}e(Y)$. A~fortiori, 
we have $\mathfrak u_{\alpha_i}\leq_{\Comf}e(Y)$ 
for all $i\le m$. Therefore, 
by Theorem~\ref{t: weak p-point below product}, 
for every $i\le m$ there exists $j\le n$ such that 
$
\mathfrak u_{\alpha_i}
\leq_{\RK}
\mathfrak u_{\beta_j}\,.
$ 
Since $A$ consists of pairwise $\leq_{\RK}$-incomparable 
ultrafilters, we conclude that 
$\mathfrak u_{\alpha_i}=\mathfrak u_{\beta_j}$\,, 
i.e., that $X\subseteq Y$. 

\par 
The theorem is proved. 
\end{proof}

\par 
Recall that all $R_{1+\alpha}$\,, and more generally, 
all $R_{<\alpha}$ with $1<\alpha<\omega_1$\,, form 
a~natural hierarchy of relations lying between 
$\le_{\RK}$ and~$\le_{\Comf}$ 
(see Section~\ref{sec: def}); therefore, 
Theorem~\ref{t: embedding into RK etal} holds 
with any such $R_{<\alpha}$\,, in particular, 
with any preorder~$\le_{\alpha}$\,.

\par 
As an easy consequence, we establish that $\scc\omega$ 
endowed with the Rudin--Keisler order, and in fact 
with any relation between it and the Comfort order, 
is {\it universal} for partially ordered sets of 
a~certain class, i.e., contains isomorphic copies 
of all of its elements.

\par 
Let us say that a~partially ordered set $(X,\le)$ 
has the {\it local cardinality}~$<\kappa$ iff, 
for all $x\in X$, the lower cone $\{y\in X:y\le x\}$ 
has the cardinality~$<\kappa$. The expressions like 
{\it local cardinality}~$\le\kappa$, or 
{\it locally finite}, etc., have the expected meaning. 

\begin{lm}\label{l: universality}
Let $\kappa\le\lambda^+$. Then 
$(\mathscr P_\kappa(\lambda),\subseteq)$ is universal 
for partially ordered sets of cardinality~$\le\lambda$ 
and local cardinality~$<\kappa$.
\end{lm}

\begin{proof}
A~standard argument: if $(X,\le)$ has 
the local cardinality~$<\kappa$, the map 
$f:X\to\mathscr P_\kappa(X)$ defined by letting 
$f(x):=\{y\in X:y\le x\}$ for all $x\in X$, 
isomorphically embeds $(X,\le)$ into 
$(\mathscr P_\kappa(X),\subseteq)$.
\end{proof}

\par 
Note that, if a~partially ordered set $(X,\le)$ 
is embeddable in $(\scc\omega,R)$ with some 
${R}\subseteq{\le_\Comf}$\,, then it has 
cardinality~$\le2^{\mathfrak c}$ and 
local cardinality~$\le\mathfrak c$\,, 
therefore, by Lemma~\ref{l: universality}, 
it is embeddable in the lattice 
$
(\mathscr P_{\mathfrak c^+}
(2^{\mathfrak c}),\subseteq). 
$ 

\begin{cor}\label{c: embedding into RK etal}
Let ${\le_\RK}\subseteq{R}\subseteq{\le_\Comf}$\,.
Then $(\scc\omega,R)$ is universal for partially 
ordered sets that are locally finite and have 
the cardinality not greater than~$2^{\mathfrak c}$. 
\end{cor}

\begin{proof}
Lemma~\ref{l: universality} and 
Theorem~\ref{t: embedding into RK etal}.
\end{proof}

\par 
In particular, $(\scc\omega,R_{<\alpha})$ 
has the specified universality property 
whenever $1<\alpha<\omega_1$\,. 

\begin{remark}\label{r: embedding into RF}
A~variant of this argument shows that 
there exists an isomorphic embedding~$e$ of 
the tree $((\scc\omega)^{<\omega},\subseteq)$ 
into $(\scc\omega,\le_{\RF})$. (Here 
$X^{<\alpha}:=\bigcup_{\beta<\alpha}X^\beta$, 
so for $s,t\in X^{<\alpha}$, $s\subseteq t$
iff $t$ end-extends~$s$.)
Indeed, if $A$ consists 
of $2^{\mathfrak c}$~pairwise $\le_{\RF}$-incomparable 
ultrafilters (even not necessarily weak $P$-points), 
then letting for all $s\in A^{<\omega}$, 
$
e(s):=
\mathfrak u_{s(0)}
\otimes
\ldots
\otimes
\mathfrak u_{s(n)}
$
where $n+1:=\dom s$, we see that 
$s\subseteq t$ implies $e(s)\le_{\RF}e(t)$, 
and the converse implication holds since 
the preorder $\le_{\RF}$ is tree-like. Consequently, 
$(\scc\omega,\le_{\RF})$ is universal for locally 
finite trees of cardinality~$\le2^{\mathfrak c}$. 
\end{remark}


\section{Embedding of the lattice 
$(\mathscr P_{\omega_1}(\omega_1),\subseteq)$}%
\label{sec: second main thm}

\par 
In this section, we prove our second main result,  
Theorem~\ref{t: embedding of countable into RK etal},
which embeds the lattice of countable subsets of 
$\omega_1$ ordered by inclusion in $\scc\omega$ with 
the Rudin--Keisler preorder and, in fact, with 
any relation between it and the Comfort preorder. 
In some ways, the reasoning generalizes those 
from the previous section. 

\par 
Let $\Lim(\omega_1)$ denote the set of nonzero 
limit ordinals $\alpha<\omega_1$\,. Pick a~so-called 
{\it ladder system} for $\Lim(\omega_1)$, in which 
thus for each $\alpha\in\Lim(\omega_1)$ there is fixed 
a~certain $\omega$-sequence $(\alpha_i)_{i\in\omega}$ 
that is increasing and cofinal in~$\alpha$. 
Let 
$
\bm{\mathfrak u}:=
(\mathfrak u_\alpha)_{\alpha<\omega_1}
$ 
be an $\omega_1$-sequence of ultrafilters 
in $\scc\omega$. For each $\alpha<\omega_1$ 
define a~map 
$f_{\bm{\mathfrak u},\alpha}:\omega\to\scc\omega$ 
and the ultrafilter  
$$
\mathfrak v_{\bm{\mathfrak u},\alpha}:=
\widetilde{f_{\bm{\mathfrak u},\alpha}}
(\mathfrak u_{\alpha})
$$ 
by recursion on~$\alpha$ as follows:  
\begin{enumerate}
\item[(i)]
$f_{\bm{\mathfrak u},0}$ is any injective map 
with the values in~$\omega$, so 
$\mathfrak v_{\bm{\mathfrak u},0}\isom\mathfrak u_0$ 
(it suffices to let $f_{\bm{\mathfrak u},0}:=\id_\omega$\,, 
so $\mathfrak v_{\bm{\mathfrak u},0}=\mathfrak u_0$);
\item[(ii)] 
if $\alpha=\beta+1$, then $f_{\bm{\mathfrak u},\alpha}$ 
is any injective map of $\omega$ into~$\scc\omega$ with 
a~strongly discrete range satisfying 
$
f_{\bm{\mathfrak u},\alpha}(i)
\isom
\mathfrak v_{\bm{\mathfrak u},\beta}
$ 
for all $i<\omega$\,;
\item[(iii)] 
if $\alpha$ is a~limit ordinal, 
then $f_{\bm{\mathfrak u},\alpha}$ is 
any injective map of $\omega$ into~$\scc\omega$ 
with a~strongly discrete range satisfying 
$
f_{\bm{\mathfrak u},\alpha}(i)
\isom
\mathfrak v_{\bm{\mathfrak u},\alpha_i}
$ 
for all $i<\omega$\,. 
\end{enumerate} 
(Let us emphasize that, although a~choice of maps 
$f_{\bm{\mathfrak u},\alpha}$ satisfying 
the specified properties can be arbitrary for each 
sequence~$\bm{\mathfrak u}$\,, our ladder system is 
fixed in advance for all such functions and sequences.) 

\begin{remark}\label{r: barriers}
Alternatively, instead of fixed ladder systems, 
the limit step~$\alpha$ of the construction can be 
handled by using uniform barriers of rank~$\alpha$ 
on~$\omega$. The idea of using barriers as 
the base sets for iterated Fubini products, due 
to Todor\v{c}evi\'{c}, was described by Dobrinen 
in Section~3.8 of~\cite{Dobrinen 2021}, and with 
more detail in Section~3 of~\cite{Dobrinen 2020}.    
\end{remark}

\par 
It is easy to see that for every $\beta<\omega_1$ 
we have: 
$$
\mathfrak v_{\bm{\mathfrak u},\beta+1}
\isom 
\mathfrak u_{\beta+1}
\otimes
\mathfrak v_{\bm{\mathfrak u},\beta}\,. 
\eqno{(\ast)}
$$
In particular, if $n<\omega$, we have 
$
\mathfrak v_{\bm{\mathfrak u},n+1}\isom 
\mathfrak u_{n+1}\otimes\ldots\otimes\mathfrak u_0\,. 
$

\begin{lm}\label{lm: u leq v}
Let 
$
\bm{\mathfrak u}:=
(\mathfrak u_\alpha)_{\alpha<\omega_1}
$ 
be an $\omega_1$-sequence of ultrafilters in 
$\scc\omega$ where all the ultrafilters 
$\mathfrak u_\alpha$ with limit indices 
$\alpha<\omega_1$ are nonprincipal. Then 
for every $\beta\leq\alpha<\omega_1$ we have:  
$$
\mathfrak u_{\beta}
\le_{\RK}
\mathfrak v_{\bm{\mathfrak u},\alpha}\,.
$$
\end{lm}

\begin{proof}
By induction on~$\alpha$. 
The induction basis ($\alpha=0$) is obvious.   
	
\par 
Let $\alpha>0$ and assume the induction hypothesis 
has already been proved for all $\beta<\alpha$. 
If $\alpha=\beta+1$, 
use~$(\ast)$ and the induction hypothesis. Assume 
now $\alpha$ is a~limit ordinal and $\beta<\alpha$. 
Then there exists an index $i_0\in\omega$ such that 
$\beta\leq\alpha_j$ for all~$j$, $i_0\leq j<\omega$. 
By the induction hypothesis, for all~$j$, 
$i_0\leq j<\omega$, we have:  
$$
\mathfrak u_{\beta}
\leq_{\RK}
\mathfrak v_{\bm{\mathfrak u},\alpha_j}
=
f_{\bm{\mathfrak u},\alpha}(j)\,.
$$ 
Since the ultrafilter $\mathfrak u_\alpha$ 
is nonprincipal, we get: 
$$
\mathfrak u_{\beta}
\leq_{\RK}
\widetilde{f_{\bm{\mathfrak u},\alpha}}
(\mathfrak u_\alpha)=
\mathfrak v_{\bm{\mathfrak u},\alpha}\,.
$$
Assume, finally, $\alpha=\beta$ is a~limit ordinal. 
Then (denoting by $\id{X}$ the identity map on~$X$) 
for all $i\in\omega$ we have:
$$
\id_{\omega}(i)
\leq_{\RK}
f_{\bm{\mathfrak u},\alpha}(i)\,,
$$
whence (since 
$\widetilde{\id_{\omega}}=\id_{\scc\omega}$) 
we get:  
$$
\mathfrak u_\alpha=
\widetilde{\id_{\omega}}(\mathfrak u_\alpha)
\leq_{\RK}
\widetilde{f_{\bm{\mathfrak u},\alpha}}
(\mathfrak u_\alpha)=
\mathfrak v_{\bm{\mathfrak u},\alpha}\,.
$$
The proof is complete. 
\end{proof}

\par 
Given an $\omega_1$-sequence 
$
\bm{\mathfrak u}:=
(\mathfrak u_\alpha)_{\alpha<\omega_1}
$ 
of ultrafilters in $\scc\omega$ and 
$X\subseteq\omega_1$\,, let 
$
\bm{\mathfrak u}^X:=
(\mathfrak u^{X}_\alpha)_{\alpha<\omega_1}
$ 
denote its $X$-{\it modification} defined by letting 
$$
\mathfrak u^{X}_\alpha:=
\begin{cases}
\mathfrak u_\alpha
&\text{if }\alpha\in X\cup\Lim(\omega_1), 
\\
0
&\text{otherwise}.
\end{cases}
$$
For better readability of the notation, 
let henceforth
$$
f_{X,\bm{\mathfrak u},\alpha}:=
f_{\bm{\mathfrak u}^{X}\!,\,\alpha}
\;\text{ and }\;
\mathfrak v_{X,\bm{\mathfrak u},\alpha}:=
\mathfrak v_{\bm{\mathfrak u}^{X}\!,\,\alpha}\,.
$$

\begin{lm}\label{lm: subseteq to leq_RK}
Let 
$
\bm{\mathfrak u}:=
(\mathfrak u_\alpha)_{\alpha<\omega_1}
$ 
be an $\omega_1$-sequence of ultrafilters 
in $\scc\omega\setminus\omega$. Then, for all 
$\beta\leq\alpha<\omega_1$ and $X,Y\subseteq\omega_1$\,,
$$
X\subseteq Y
\;\text{ implies }\;
\mathfrak v_{X,\bm{\mathfrak u},\beta}
\leq_{\RK}
\mathfrak v_{Y,\bm{\mathfrak u},\alpha}\,.
$$
\end{lm}

\begin{proof}
By induction on~$\alpha$. For the induction basis, 
let $\alpha=0$ and $X\subseteq Y\subseteq\omega_1$\,. 
Then also $\beta=0$. If $0\in X$, then 
$
\mathfrak v_{X,\bm{\mathfrak u},0}
\isom
\mathfrak u_0
\isom
\mathfrak v_{Y,\bm{\mathfrak u},0}\,,
$ 
so the required statement holds. If $0\notin X$, 
then $\mathfrak v_{X,\bm{\mathfrak u},0}\in\omega$, 
and the required statement holds as well. 
	
\par 
Suppose now $\alpha>0$, let 
$X\subseteq Y\subseteq\omega_1$\,, and assume 
that the induction hypothesis has already been 
proved for all $\alpha'<\alpha$.

Let $\alpha=\gamma+1$. Then by~$(\ast)$ we have:  
$$
\mathfrak v_{Y,\bm{\mathfrak u},\alpha}
\isom
\mathfrak u^{Y}_\alpha
\otimes
\mathfrak v_{Y,\bm{\mathfrak u},\gamma}\,.
$$
If $\beta<\alpha$ then $\beta\leq\gamma$, 
so by the induction hypothesis we have: 
$$
\mathfrak v_{X,\bm{\mathfrak u},\beta}
\leq_{\RK}
\mathfrak v_{Y,\bm{\mathfrak u},\gamma}
\leq_{\RK}
\mathfrak u^{Y}_\alpha
\otimes 
\mathfrak v_{Y,\bm{\mathfrak u},\gamma}
\isom
\mathfrak v_{Y,\bm{\mathfrak u},\alpha}\,.
$$

And if $\beta=\alpha$, 
since 
$
\mathfrak v_{X,\bm{\mathfrak u},\gamma}
\le_{\RK}
\mathfrak v_{Y,\bm{\mathfrak u},\gamma}
$ 
by the induction hypothesis, we have 
$$
\mathfrak v_{X,\bm{\mathfrak u},\beta}=
\mathfrak v_{X,\bm{\mathfrak u},\alpha}
\isom
\mathfrak u_{\alpha}\otimes 
\mathfrak v_{X,\bm{\mathfrak u},\gamma}
\le_{\RK}
\mathfrak u_{\alpha}\otimes 
\mathfrak v_{Y,\bm{\mathfrak u},\gamma}
\isom
\mathfrak v_{Y,\bm{\mathfrak u},\alpha}\,. 
$$ 

Let $\alpha>0$ be a~limit ordinal. 
If $\beta<\alpha$, then there exists $i_0\in\omega$ 
such that $\beta\leq\alpha_j$ for all~$j$, 
$i_0\leq j<\omega$. By the induction hypothesis, 
for all $j$, $i_0\leq j<\omega$, we have:  
$$
\mathfrak v_{X,\bm{\mathfrak u},\beta}
\leq_{\RK} 
\mathfrak v_{Y,\bm{\mathfrak u},\alpha_j}=
f_{Y,\bm{\mathfrak u},\alpha}(j)\,.
$$ 
Whence, since the ultrafilter 
$\mathfrak u_\alpha$ is nonprincipal, we get: 
$$
\mathfrak v_{X,\bm{\mathfrak u},\beta}
\leq_{\RK}
\widetilde{f_{Y,\bm{\mathfrak u},\alpha}}
(\mathfrak u_\alpha)=
\mathfrak v_{Y,\bm{\mathfrak u},\alpha}\,.
$$
And if $\alpha=\beta$, then for each $i<\omega$ 
by the induction hypothesis we have: 
$$
f_{X,\bm{\mathfrak u},\alpha}(i)=
\mathfrak v_{X,\bm{\mathfrak u},\alpha_i}
\leq_{\RK}
\mathfrak v_{Y,\bm{\mathfrak u},\alpha_i}=
f_{Y,\bm{\mathfrak u},\alpha}(i)\,.
$$
It follows 
$$
\mathfrak v_{X,\bm{\mathfrak u},\alpha}=
\widetilde{f_{X,\bm{\mathfrak u},\alpha}}
(\mathfrak u_\alpha)
\leq_{\RK} 
\widetilde{f_{Y,\bm{\mathfrak u},\alpha}}
(\mathfrak u_\alpha)=
\mathfrak v_{Y,\bm{\mathfrak u},\alpha}\,.
$$
The proof is complete. 
\end{proof}

\par 
The following lemma can be considered 
as a~generalization of 
Theorem~\ref{t: weak p-point below product}.

\begin{lm}\label{lm: p-point leq v} 
Let 
$
\bm{\mathfrak u}:=
(\mathfrak u_\alpha)_{\alpha<\omega_1}
$ 
be an $\omega_1$-sequence of ultrafilters in 
$\scc\omega\setminus\omega$. Then for every 
$\alpha<\omega_1$\,, $X\subseteq\omega_1$\,, 
and weak $P$-point~$\mathfrak w$, 
$$
\mathfrak w
\leq_{\Comf}
\mathfrak v_{X,\bm{\mathfrak u},\alpha}
\:\text{ implies }\;
\mathfrak w
\leq_{\RK}
\mathfrak u_\beta
$$
for some 
$\beta\in(X\cup\Lim(\omega_1))\cap(\alpha+1)$.  
\end{lm}

\begin{proof} 
As in the proof of 
Theorem~\ref{t: weak p-point below product}, 
let us start by recalling that, whenever $\mathfrak w$ 
is a~weak $P$-point and $\mathfrak v$ an arbitrary 
ultrafilter, then $\mathfrak w\leq_{\Comf}\mathfrak v$ 
implies $\mathfrak w\leq_{\RK}\mathfrak v$ 
(see~\cite{Garcia-Ferreira 1993}, Theorem~2.10). 
Now prove the lemma by induction on~$\alpha$. 

\par 
The induction basis ($\alpha=0$) is evident. 
Let $\alpha>0$, assume the induction hypothesis 
holds for all $\beta<\alpha$, and 
$
\mathfrak w
\leq_{\RK}
\mathfrak v_{X,\bm{\mathfrak u},\alpha}=
\widetilde{f_{X,\bm{\mathfrak u},\alpha}}
(\mathfrak u^{X}_\alpha)\,,
$ 
i.e., for some $g:\omega\to\omega$ we have 
$$
\mathfrak w=
\widetilde{g}(\widetilde{f_{X,\bm{\mathfrak u},\alpha}}
(\mathfrak u^{X}_\alpha))=
\widetilde{g\circ f_{X,\bm{\mathfrak u},\alpha}}
(\mathfrak u^{X}_\alpha).
$$
Since $\mathfrak w$ is a~weak $P$-point, 
it follows from Theorem~\ref{t: weak p-point}(iib) 
that one of two sets 
$
A:=
\{i\in\omega: 
g\circ f_{X,\bm{\mathfrak u},\alpha}(i)
=\mathfrak w\} 
$
and
$ 
B:=
\{i\in\omega: 
g\circ f_{X,\bm{\mathfrak u},\alpha}(i)
\in\omega\}
$ 
belongs to~$\mathfrak u^{X}_\alpha$.

\par 
If $A\in\mathfrak u^{X}_\alpha$, 
pick any $i\in A$. We have: 
$$
\mathfrak w=
\widetilde{g\circ f_{X,\bm{\mathfrak u},\alpha}}(i)=
\widetilde{g}(\widetilde{f_{X,\bm{\mathfrak u},\alpha}}(i)),
$$
where the ultrafilter 
$\widetilde{f_{X,\bm{\mathfrak u},\alpha}}(i)$ 
equals $\mathfrak v_{X,\bm{\mathfrak u},\beta}$ 
for some $\beta<\alpha$. Thus, for some $\beta<\alpha$, 
we have 
$$
\mathfrak w
\leq_{\RK}
\mathfrak v_{X,\bm{\mathfrak u},\beta}\,,
$$
and it remains to use the induction hypothesis. 

\par 
If $B\in\mathfrak u^{X}_\alpha$, 
then for all $i\in B$ we have
$$
\widetilde{g\circ f_{X,\bm{\mathfrak u},\alpha}}(i)=
\widetilde{g}
(\widetilde{f_{X,\bm{\mathfrak u},\alpha}}(i))
\in\omega,
$$ 
so we can define a~map 
$h:\omega\to\omega$ by letting 
$$
h(i):=
\begin{cases}
\widetilde{g}(\widetilde{f_{X,\bm{\mathfrak u},\alpha}}(i))
&\text{if }i\in B,
\\
0
&\text{otherwise}.
\end{cases}
$$
The maps $\widetilde h$ 
and $\widetilde g\circ\widetilde f$ are continuous 
and coincide on the set $B\in\mathfrak u^{X}_\alpha$. 
It follows 
$$
\widetilde{h}(\mathfrak u^{X}_\alpha)=
\widetilde{g}
(\widetilde{f_{X,\bm{\mathfrak u},\alpha}}
(\mathfrak u^{X}_\alpha))=
\mathfrak w,
$$
hence, 
$\mathfrak w\leq_{\RK}\mathfrak u^{X}_\alpha\,.$ 
In particular, it follows that the ultrafilter 
$\mathfrak u^{X}_\alpha$ is nonprincipal. Therefore, 
$\alpha\in\Lim(\omega_1)$ or 
$\mathfrak u^{X}_\alpha=\mathfrak u_\alpha$\,, 
which thus completes the proof.  
\end{proof} 

\begin{lm}\label{lm: leq_RK to subseteq}
Let
$
\bm{\mathfrak u}:=
(\mathfrak u_\alpha)_{\alpha<\omega_1}
$ 
be an $\omega_1$-sequence of weak $P$-points 
in~$\scc\omega$. Then for all $\alpha,\beta<\omega_1$ 
and $X,Y\subseteq\omega_1$\,, 
if
$$
\mathfrak v_{X,\bm{\mathfrak u},\alpha}
\leq_{\Comf}
\mathfrak v_{Y,\bm{\mathfrak u},\beta}\,,
$$ 
then for each $\gamma\in X\cap(\alpha+1)$ 
there exists 
$
\delta\in
(Y\cup\Lim(\omega_1))\cap(\beta+1) 
$ 
such that 
$$
\mathfrak u_\gamma
\leq_{\RK}
\mathfrak u_\delta\,.
$$
\end{lm}

\begin{proof}
Note that $\mathfrak u^{X}_\alpha=\mathfrak u_\alpha$ 
for every limit ordinal~$\alpha$. Therefore, 
in~$\bm{\mathfrak u}^X$, all ultrafilters 
$\mathfrak u^{X}_\alpha$ with limit indices~$\alpha$ 
are nonprincipal, and thus we are under the assumption 
of Lemma~\ref{lm: u leq v}. It follows that, 
for each $\gamma\in X\cap(\alpha+1)$, we have:  
$$
\mathfrak u_\gamma=
\mathfrak u^{X}_\gamma
\leq_{\RK}
\mathfrak v_{X,\bm{\mathfrak u},\alpha}
\leq_{\Comf}
\mathfrak v_{Y,\bm{\mathfrak u},\alpha}
$$
(where $\mathfrak u_\gamma=\mathfrak u^{X}_\gamma$ 
holds since $\gamma\in X$ and 
$
\mathfrak u_\gamma
\leq_{\RK}
\mathfrak v_{X,\bm{\mathfrak u},\alpha}
$ 
by Lemma~\ref{lm: u leq v}), and so,
$$
\mathfrak u_\gamma
\leq_{\Comf}
\mathfrak v_{Y,\bm{\mathfrak u},\alpha}\,.
$$
But then, since $\mathfrak u_\gamma$ is a~weak 
$P$-point, by using Lemma~\ref{lm: p-point leq v} 
we get $\mathfrak u_\gamma\leq_\RK\mathfrak u_\delta$ 
for some $\delta\in(Y\cup\Lim(\omega_1))\cap(\beta+1),$
as required. 
\end{proof}

\par 
Combining the preceding lemmas, we obtain 
the main result of this section: 

\begin{thm}[The Second Main Theorem]%
\label{t: embedding of countable into RK etal}
There exists a~map 
$e:\mathscr P_{\omega_1}(\omega_1)\to\scc\omega$ 
such that for all $X,Y\in\mathscr P_{\omega_1}(\omega_1)$, 
\begin{enumerate}
\item[(i)] 
if $X\subseteq Y$ then $e(X)\le_\RK e(Y)\,;$
\item[(ii)] 
if $e(X)\leq_{\Comf}e(Y)$ then $X\subseteq Y$. 
\end{enumerate}
Consequently, $e$~is an isomorphic embedding of the 
lattice $(\mathscr P_{\omega_1}(\omega_1),\subseteq)$
into the structure $(\scc\omega,R)$ whenever 
${\le_\RK}\subseteq{R}\subseteq{\le_\Comf}$\,. 
\end{thm}

\begin{proof}
Let 
$
\bm{\mathfrak u}:=
(\mathfrak u_\alpha)_{\alpha<\omega_1}
$
be an $\omega_1$-sequence of pairwise 
$\le_{\RK}$-incomparable weak $P$-points 
in~$\scc\omega$, and let 
$A:=\omega_1\setminus\Lim(\omega_1)$. 
Since $|A|=\aleph_1$\,, it suffices to define 
a~required map~$e$ whose domain is rather 
$\mathscr P_{\omega_1}(A)$ than 
$\mathscr P_{\omega_1}(\omega_1)$. 

\par 
If $X\subseteq\omega_1$ is at most countable, 
we let $\sup^+X:=\min(\omega_1\setminus X)$. Define 
$e$ by letting, for all $X\in\mathscr P_{\omega_1}(A)$, 
$$
e(X):=
\mathfrak v_{X,\bm{\mathfrak u},\sup^+\!X}\,.
$$ 
Let us verify that, for all 
$X,Y\in\mathscr P_{\omega_1}(A)$, 
both (i) and~(ii) hold. 

\par 
(i). 
If $X\subseteq Y$, then $\sup^+X\subseteq\sup^+Y$,  
whence it follows $e(X)\le_{\RK}e(Y)$ 
by Lemma~\ref{lm: subseteq to leq_RK}. 

\par 
(ii). 
Assume $e(X)\le_{\Comf}e(Y)$.  
Since $Z\subseteq\sup^+Z$ for every 
at most countable $Z\subseteq\omega_1$\,, 
by Lemma~\ref{lm: leq_RK to subseteq},  
for each $\gamma\in X$ there exists 
$\delta\in Y\cup\Lim(\omega_1)$ such that  
$\mathfrak u_\gamma\leq_{\RK}\mathfrak u_\delta$\,. 
But the sequence $\bm{\mathfrak u}$ consists of 
pairwise $\le_{\RK}$-incomparable ultrafilters, 
therefore, $\delta=\gamma$, and so 
$\gamma\in Y\cup\Lim(\omega_1)$\,. Since $\gamma$ 
is a~successor ordinal, we conclude that 
$\gamma\in Y$. Thus, finally, $X\subseteq Y$. 

\par 
The theorem is proved. 
\end{proof}

\begin{cor}\label{c: embedding of countable into RK etal}
Let ${\le_\RK}\subseteq{R}\subseteq{\le_\Comf}$\,.
Then $(\scc\omega,R)$ is universal for partially 
ordered sets that are locally at most countable 
and have the cardinality at most~$\aleph_1$\,. 
\end{cor}

\begin{proof}
By Lemma~\ref{l: universality} and 
Theorem~\ref{t: embedding of countable into RK etal}.
\end{proof}

\par 
In particular, 
Theorem~\ref{t: embedding of countable into RK etal} and 
Corollary~\ref{c: embedding of countable into RK etal}
hold for $(\scc\omega,R_{1+\alpha})$ 
for any $\alpha<\omega_1$\,. 

\begin{remark}\label{r: countable embedding into RF}
As in Remark~\ref{r: embedding into RF}, 
a~variant of the argument shows that 
there exists an isomorphic embedding~$e$ of 
the tree $({\omega_1}^{<{\omega_1}},\subseteq)$ 
into $(\scc\omega,\le_{\RF})$. 
First, we should alter the construction of 
ultrafilters $\mathfrak v_{\bm{\mathfrak u},\alpha}$ 
to guarantee that Lemma~\ref{lm: u leq v} remain 
true for $\le_{\RF}$\,. 
Then, if $s\in{\omega_1}^{\le{\omega_1}}$ and 
$
\bm{\mathfrak u}:=
(\mathfrak u_\alpha)_{\alpha<\omega_1}
$ 
is an $\omega_1$-sequence in~$\scc\omega$, define 
an {\it $s$-modification} $\bm{\mathfrak u}^{s}$ 
of~$\bm{\mathfrak u}$ by letting 
$\mathfrak u^{s}_\alpha:=\mathfrak u_{s(\alpha)}$ 
if $\alpha\in\dom s$, and $\mathfrak u^{s}_\alpha:=0$ 
otherwise, and verify that the corresponding variant 
of Lemma~\ref{lm: subseteq to leq_RK} remains true 
and that, with $e$~appropriately defined, 
$s\subseteq t$ implies $e(s)\le_{\RF}e(t)$ 
for all $s,t\in{\omega_1}^{<{\omega_1}}$. 
But then the converse implication holds 
because the $\mathfrak u_\alpha$ are assumed 
to be $\le_{\RF}$-incomparable and the preorder 
$\le_{\RF}$ is tree-like, so $e$~is a~required 
embedding. Consequently, $(\scc\omega,\le_{\RF})$ 
is universal for locally countable (well-founded) 
trees of cardinality~$\le\aleph_1$\,. 
\end{remark}

\section{Concluding remarks}\label{sec: conclusion} 

\par 
An obvious necessary condition for an embeddability 
of a~partially ordered set $(X,\le)$ in 
$(\scc\omega,\le_{\RK})$, or more generally, 
in $(\scc\omega,R)$ with 
${\le_\RK}\subseteq{R}\subseteq{\le_\Comf}$\,, is 
that $(X,\le)$ be of cardinality~$\le2^{\mathfrak c}$ 
and local cardinality~$\le\mathfrak c$\,. The second 
author is inclined to believe that this condition is 
also sufficient and that, moreover, this criterion 
is provable in $\ZFC$ alone. The results of this 
paper are not nearly as strong; we were able to 
prove in $\ZFC$ only that each of two lattices 
$(\mathscr P_{\omega}(2^{\mathfrak c}),\subseteq)$ 
and $(\mathscr P_{\omega_1}(\omega_1),\subseteq)$ 
is embeddable in $(\scc\omega,R)$. 
Minor improvements in these results can be made without 
much effort. E.g., using $\mathfrak c^+$-directedness 
of~$\le_{\RK}$\,, it is easy to construct by recursion 
on $\alpha<\mathfrak c^+$ an embedding of the 
lexicographic product of the ordinal~$\mathfrak c^+$ and 
the lattice $(\mathscr P_{\omega_1}(\omega_1),\subseteq)$ 
in $(\scc\omega,\le_{\RK})$ (the same fact  
could be stated for any $(\scc\omega,R)$ with 
${\le_\RK}\subseteq{R}\subseteq{\le_\Comf}$ if 
we knew that there exist $2^{\mathfrak c}$ weak 
$P$-points $\le_\RK$-above every ultrafilter). 
Perhaps, some modification of the method presented 
here might allow us to embed the lattice
$(\mathscr P_{\omega_1}(2^{\mathfrak c}),\subseteq)$ 
in $(\scc\omega,\le_{\RK})$. However, it seems that 
new ideas will be needed to move forward even further.


\begin{footnotesize} 
\noindent
{\sc
HSE University
\/}
\\
{\it E-mail address:\/} 
niknikols0@gmail.com
\end{footnotesize} 

\begin{footnotesize} 
\noindent
{\sc
Higher School of Modern Mathematics MIPT
\/}
\\
{\it E-mail address:\/} 
d.i.saveliev@gmail.com 
\end{footnotesize} 


\begin{thebibliography}{111}



\hide{
\bibitem{Bartoszynski Shelah 2003}
Tomek Bartoszy{\'n}ski, Saharon Shelah, 
{\it There may be no Hausdorff ultrfailters},
2003, arXiv:math/0311064.
}

\hide{
\bibitem{Bergelson etal 2008}
Vitaly Bergelson, Andreas R.~Blass, 
Mauro Di Nasso, Renling Jin (editors),
{\em Ultrafilters across Mathematics},
International Congress UltraMath 2008: 
Applications of Ultrafilters and Ultraproducts 
in Mathematics, 
Contemporary Mathematics 530, 
The American Mathematical Society, 
2010. 
}

\bibitem{Bernstein 1970} 
Allen R.~Bernstein, 
{\em A~new kind of compactness for topological spaces},
Fundamenta Mathematic{\ae}
66 (1970), 185--193.

\bibitem{Blass 1970} 
Andreas R.~Blass,
{\it Ordering of ultrafilters}, 
PhD Thesis, Harvard University, Cambridge, Mass., 1970.

\bibitem{Blass 1973} 
Andreas R.~Blass, 
{\em The Rudin--Keisler ordering of P-points},
Transactions of the American Mathematical Society
179 (1973), 145--166.

\bibitem{Blass 1981} 
Andreas R.~Blass,
{\it Some initial segments of 
the Rudin--Keisler ordering}, 
J.~Symbolic Log., 46:1 (1981), 147--157.

\bibitem{Blass 1988} 
Andreas R.~Blass,
{\it Selective ultrafilters and homogeneity}, 
Annals of Pure and Applied Logic, 38 (1988), 215--255.

\bibitem{Blass Shelah 1987} 
Andreas R.~Blass, Saharon Shelah, 
{\it There may be simple $P_{\aleph_1}$-points and 
$P_{\aleph_2}$-points and the Rudin--Keisler ordering 
may be downward directed},
Annals of Mathematical Logic
33 (1987), 213--243.

\bibitem{Booth 1969}  
David D.~Booth, 
{\it Countably indexed ultrafilters}, 
Doctoral dissertation, University of Wisconsin, 1969.

\bibitem{Booth 1970}  
David D.~Booth, 
{\it Ultrafilters on a~countable set},
Annals of Mathematical Logic
2 (1970), 1--24.

\bibitem{Brauninger Mildenberger 2023}
Christian Bräuninger, Heike Mildenberger, 
{\it A simple $P_{\aleph_1}$-point and 
a~simple $P_{\aleph_2}$-point},
Journal of the European Mathematical Society, 
25:12 (2023), 4971--4996.

\hide{
\bibitem{Brian 2020} 
Will Brian, 
{\it The isomorphism class of the shift map}, 
Topology and its Applications, 283 (2020), 107343, 16. 

\bibitem{Brian 2024} 
Will Brian, 
{\it Does $\mathcal P(\mathbb N)/\mathrm{fin}$ 
know its right hand from its left?}, 
arXiv:2402.04358v2 [math.LO].
}

\hide{
\bibitem{Butkovicova} 
Eva Butkovi\v{c}ov\'{a}, 
{\em Decreasing chains without lower bounds 
in the Rudin--Frol{\'\i}k order}, 
Proceedings of the American Mathematical Society
109:1 (1990), 251--259.
}


\hide{
\bibitem{Chang Keisler} 
C.\,C.~Chang, H.~Jerome Keisler,
{\em Model Theory},
3rd ed., North-Holland, Amsterdam--New York, 1990. 
}

\bibitem{Choquet 1968a} 
Gustave Choquet, 
{\it Construction d'ultrafiltres sur~$\mathbb N$}, 
Bulletin Sci. Math. 92:2 (1968), 41--48.

\bibitem{Choquet 1968b} 
Gustave Choquet, 
{\it Deux classes remarquables d'ultrafiltres 
sur~$\mathbb N$},
Bulletin Sci. Math. 92:2 (1968), 143--153.

\bibitem{Comfort 1977} 
W.\,Wistar Comfort, 
{\it Ultrafilters: some old and some new results},
Bulletin of the American Mathematical Society,
83:4 (1977), 417--455.

\bibitem{Comfort Negrepontis 1972} 
W.\,Wistar Comfort, Stylianos Negrepontis,
{\it On families of large oscillation}, 
Fundamenta Mathematicae 75 (1972), 275--290.

\bibitem{Comfort Negrepontis 1974} 
W.\,Wistar Comfort, Stylianos Negrepontis,
{\em The theory of ultrafilters},
Springer-Verlag, Berlin, 1974.


\bibitem{Daguenet 1975} 
Maryvonne Daguenet, 
{\it Rapport entre l'ensemble des ultrafiltres 
admettant un ultrafiltre donn\'{e} pour image 
et l'ensemble des images de cet ultrafiltre}, 
Commentationes Mathematicae Unirersitatis Carolinae, 
16 (1975), 99--113.

\bibitem{Daguenet 1979} 
Maryvonne Daguenet-Teissier, 
{\it Ultrafiltres {\'a}~la fa\c{c}on de Ramsey}, 
Transactions of the American Mathematical Society,
250 (1979), 91--120.

\hide{
\bibitem{DiNasso Forti 2006}
Mauro Di Nasso, Marco Forti, 
{\it Hausdorff ultrafilters},
Proceedings of the AMS, 
134:6 (2006), 1809--1818.
}

\hide{
\bibitem{Dobrinen 2014} 
Natasha Dobrinen, 
{\it Survey on the Tukey ordering of ultrafilters}, 
2014, arXiv:1401.8108.

\bibitem{Dobrinen Todorcevic 2011}  
Natasha Dobrinen, Stevo Todorcevic, 
{\it Tukey types of ultrafilters}, 
Illinois J.~Math. 55:3 (2011), 907--951. 
}


\bibitem{Dobrinen 2020} 
Natasha Dobrinen, 
{\it Continuous and other finitely generated 
canonical cofinal maps on ultrafilters}, 
Fundamenta Mathematicae, 249:2 (2020), 111--147.


\bibitem{Dobrinen 2021} 
Natasha Dobrinen, 
{\it Topological Ramsey spaces dense in forcings}, 
3--58 in: D.~Cenzer, Ch.~Porter, J.~Zapletal (eds.), 
{\it Structure and randomness in computability 
and set theory}, 
World Scientific Publ., Hackensack, NJ, 2021. 


\hide{
\bibitem{vanDouwen 1990}
Eric K.~van Douwen, 
{\it Martin's axiom and pathological points 
in $\scc X\setminus X$},
Topology and its Applications 34 (1990), 3--33.
}

\hide{
\bibitem{Dow 1985} 
Alan Dow, 
{\em Good and OK ultrafilters},
Transactions of the American Mathematical Society
290 (1985), 145--160.
}


\hide{
\bibitem{Engelking} 
Ryszard Engelking,
{\it General topology},
Monografie matematyczne, vol.~60, 
PWN-Polish Scientific Publishers, Warszawa, 1977;
2nd ed., Sigma series in pure mathematics, vol.~6, 
Heldermann, 1989.
}


\bibitem{Frayne Morel Scott 1962-3}
Thomas Frayne, Anne C.~Morel, Dana S.~Scott, 
{\it Reduced direct product},
Fundamenta Mathematicae, 
51, (1962) 195--228; 53 (1963),~117.

\hide{
\bibitem{Frolik 1960} 
Zdenek Frol{\'\i}k,
{\em On the topological product of paracompact spaces},
Bull. Acad. Polon., 8 (1960), 747-750.
}

\bibitem{Frolik 1967a} 
Zdenek Frol{\'\i}k,
{\it Sums of ultrafilters}, 
Bulletin of the Amer. Math. Soc. 73 (1967), 87--91. 

\bibitem{Frolik 1967b} 
Zdenek Frol{\'\i}k,
{\it Non-homogeneity of $\scc P-P$},
Commentationes Mathematicae Universitatis Carolinae, 
8:4 (1967), 705--709. 

\bibitem{Frolik 1967c} 
Zdenek Frol{\'\i}k,
{\it Homogeneity problems for 
extremally disconnected spaces},
Commentationes Mathematicae Universitatis Carolinae, 
8:4 (1967), 757--763. 


\bibitem{Garcia-Ferreira 1993} %
Salvador Garc{\'\i}a-Ferreira, 
{\em Three orderings on $\scc(\omega)\setminus\omega$},
Topology and its Applications 50 (1993), 199--216.

\bibitem{Garcia-Ferreira 1994}  
Salvador Garc{\'\i}a-Ferreira, 
{\em Comfort types of ultrafilters},
Proceedings of the American Mathematical Society
120 (1994), 1251--1260.

\hide{
\bibitem{Garcia-Ferreria Hindman Strauss}  
Salvador Garc{\'\i}a-Ferreira, 
Neil Hindman, Dona Strauss, 
{\em Orderings of the Stone–{\v C}ech 
remainder of a~discrete semigroup},
Topology and its Applications 97 (1999) 127--148. 
}
 
\hide{
\bibitem{Garcia-Ferreira Ortiz-Castillo 2015}  
Salvador Garc{\'\i}a-Ferreira, 
Y.~F.~Ortiz-Castillo,
{\em The subspace of weak P-points of $\mathbb N^*$},
Commentationes Mathematicae Universitatis Carolinae 
56:2 (2015), 231--236.
}

\bibitem{Ginsburg Saks} 
John Norman Ginsburg, Viktor Saks, 
{\em Some applications of ultrafilters in topology}, 
Pacific Journal of Mathematics 57 (1975), 403--418. 

\hide{
\bibitem{Goranko} 
Valentin Goranko,
{\em Filter and ultrafilter extensions of structures: 
universal-algebraic aspects},
Preprint, 2007. 
}


\bibitem{Hart vanMill 2024} 
Klaas Pieter Hart, Jan van Mill,
{\it Problems on $\scc\mathbb N$},
Topology and its Applications, 364:1.
arXiv:2205.11204.

\bibitem{Hindman Strauss} 
Neil Hindman, Dona Strauss,
{\em Algebra in the Stone--\v{C}ech compactification},
2nd~ed., revised and expanded, 
W.~de~Gruyter, Berlin--N.Y., 2012.

\hide{
\bibitem{Hindman Strauss 2017} 
Neil Hindman, Dona Strauss,
{\it Topological properties of some algebraically 
defined subsets of~$\scc\mathbb N$},
Topology and its Application, 220 (2017), 43--49.

\bibitem{Hindman Strauss non-Borel} 
Neil Hindman, Dona Strauss,
{\it Sets and mapping in $\scc S$ 
which are not Borel},
New York Journal of Mathematics, 
24 (2018), 689--701.
}



\hide{
\bibitem{Kanamori} 
Akihiro Kanamori. 
{\it The higher infinite:
large cardinals in set theory from their beginnings.\/}
2nd~ed., Springer-Verlag, Berlin, 2005.
}

\bibitem{Keisler 1967} 
H.~Jerome Keisler,
Mimeographed Lecture Notes, 
University of California (Los Angeles), 1967.

\hide{
\bibitem{Keisler 2010} 
H.~Jerome Keisler,
{\em On ultraproduct construction},
in: \cite{Bergelson etal 2008}, 163--179.
}

\bibitem{Kochen 1961} 
Simon Kochen,
{\em Ultraproducts in the theory of models},
Annals of Mathematics 74:2 (1961), 221--261.

\bibitem{Kunen 1970}
Kenneth Kunen, 
{\it On the compactification of the integers}, 
Notices Amer. Math. Soc. 17 (1970), 299. 
Abstract~70T-G7.

\bibitem{Kunen 1972}
Kenneth Kunen, 
{\it Ultrafilters and independent sets}, 
Transactions of the American Mathematical Society,
172 (1972), 299--306.

\bibitem{Kunen 1978}
Kenneth Kunen, 
{\it Weak P-points in $N^*$},
Colloquia Mathematica Societatis J{\'a}nos Bolyai, 
23.~Topology (1978), 741--749.


\bibitem{Laflamme 1989}
Claude Laflamme, 
{\it Forcing with filters and complete combinatorics}, 
Annals of Mathematical Logic 42 (1989) 125--163.


\hide{
\bibitem{Malliaris Shelah 2013} 
Maryanthe E.~Malliaris, Saharon Shelah, 
{\em General topology meets model theory, 
on $\mathfrak p$ and~$\mathfrak t$},
Proceedings of the National Academy of Sciences 
of the United States of America, 
110:33 (2013) 13300--13305.

\bibitem{Malliaris Shelah 2016} 
Maryanthe E.~Malliaris, Saharon Shelah, 
{\em Cofinality spectrum theorems in model theory, 
set theory, and general topology}, 
Journal of the American Mathematical Society, 
29:1 (2016), 237--297;
arXiv:1208.5424.
}

\bibitem{Mathias 1972}
Adrian R.\,D.~Mathias,
{\it Solution of problems of Choquet and Puritz}, 
in: Proceedings Conference in Math. Logic,
London 1970, 
Lecture Notes in Mathematics, 255, 204--210, 
Berlin-Heidelberg-New York, Springer, 1972.

\hide{
\bibitem{vanMill 1982}
Jan van Mill, 
{\it Sixteen topological types in $\scc\omega-\omega$}, 
Topology and its Applications, 13:1 (1982), 43--57. 
}

\bibitem{vanMill 1984}
Jan van Mill, 
{\it An introduction to $\scc\omega$}, 
Handbook of set-theoretic topology, 
K.~Kunen and J.\,E.~Vaughan (eds.),
North-Holland, Amsterdam, 1984,
503--567. 

\hide{
\bibitem{Miller 1980}
Arnold W.~Miller, 
{\it There are no Q-points in Laver’s model 
for the Borel conjecture}, 
Proceedings of the American Mathematical Society,
78:1 (1980), 103--106. 
}




\hide{
\bibitem{Parovicenko 1963}
Ivan I.~Parovi{\v c}enko, 
{\it A~universal bicompact of weight~$\aleph_1$}\,, 
Soviet Mathematics Doklady, 4 (1963), 592--595.
}

\bibitem{Pitt 1971}
R.\,A.~Pitt, 
{\it The classification of ultrafilters on~$\mathbb N$}, 
Doctoral dissertation, University of Leicester, 1971.

\hide{
\bibitem{Poliakov 2023}
Nikolai L.~Poliakov, 
{\it On the canonical Ramsey theorem of 
Erd{\H o}s and Rado and Ramsey ultrafilters}, 
Doklady Math., 108:2 (2023), 392--401. 
}

\bibitem{Poliakov Saveliev 2021}
Nikolai L.~Poliakov, Denis I.~Saveliev, 
{\it On ultrafilter extensions of first-order 
models and ultrafilter interpretations}, 
Archive for Mathematical Logic 60 (2021), 625--681. 

\bibitem{Poliakov Saveliev 2025}
Nikolai L.~Poliakov, Denis I.~Saveliev, 
{\it Generalizations of the Rudin–Keisler preorder 
and their model-theoretic applications}, 
Bulletin of L.\,N.~Gumilyov Eurasian National 
University, Mathematics, Computer Science, Mechanics 
series, 151:2 (2025), 6--11. 

\bibitem{Poliakov Saveliev 2026}
Nikolai L.~Poliakov, Denis I.~Saveliev, 
{\it Solution to Hart–van Mill's Problem~61}, 
Russian Math. Surveys, 81:1 (2026), 189--190. 



\bibitem{Raghavan Shelah 2017}
Dilip Raghavan, Saharon Shelah, 
{\em On embedding certain partial orders into 
the $P$-points under Rudin--Keisler and Tukey 
reducibility}, 
Transaction of the American Mathematical Society,
369:6 (2017), 4433--4455.  

\hide{
\bibitem{Raghavan Verner}
Dilip Raghavan, Jonathan Verner, 
{\em Chains of P-points}, 
arXiv:1801.02410.
}

\bibitem{WRudin 1956}
Walter Rudin, 
{\it Homogeneity problems in the theory 
of {\v C}ech compactification}, 
Duke Math.~J., 23 (1956), 409--419, 633.

\bibitem{MERudin 1966}
Mary Ellen Rudin, 
{\it Types of ultrafilters}, 
in: Topology Seminar, Wisconsin, 1965. 
Annals of Mathematics Studies, 60, 147--151. 
Princeton: Princeton University Press, 1966.

\bibitem{MERudin 1971}
Mary Ellen Rudin, 
{\it Partial orders on the types in $\scc N$}, 
Transaction of the American Mathematical Society,
155:2 (1971), 353--362.


\bibitem{Saveliev 2011} 
Denis I.~Saveliev,
{\em Ultrafilter extensions of models},
Lecture Notes in Computer Science 6521 (2011), 162--177. 

\bibitem{Saveliev 2012} 
Denis I.~Saveliev,
{\em On ultrafilter extensions of models},
in: The Infinity Project Proceedings, 
S.-D.~Friedman, M.~Koerwien, M.\,M.~M{\"u}ller (eds.), 
CRM~Documents~11, Barcelona, 2012, 599--616.

\hide{
\bibitem{Saveliev(idempotents)} 
Denis I.~Saveliev,
{\it On idempotents in compact left topological 
universal algebras}, 
Topology Proceedings 43 (2014), 37--46.
}

\hide{

\bibitem{Saveliev(orders)} 
Denis I.~Saveliev,
{\it Ultrafilter extensions of linearly ordered sets},
Order 32:1 (2015), 29--41. 

\bibitem{Saveliev(2 concepts)} 
Denis I.~Saveliev,
{\it On two types of ultrafilter extensions 
of binary relations},
Preprint, arXiv:~2001.02456.
}

\hide{
\bibitem{Saveliev Shelah} 
Denis I.~Saveliev, Saharon Shelah,
{\it Ultrafilter extensions do not preserve 
elementary equivalence}, 
Mathematical Logic Quarterly, 
65:4 (2019), 511--516. 
}

\bibitem{Shelah forcing} 
Saharon Shelah, 
{\em Proper and improper forcing},
2nd ed., Perspectives in Mathematical Logic, 
Springer-Verlag, Berlin, 1998. 

\bibitem{Shelah Rudin 1978}
Saharon Shelah, Mary Ellen Rudin, 
{\it Unordered types of ultrafilters}, 
Topology Proceedings, 3:1 (1978), 199--204.

\hide{
\bibitem{Shelah Steprans 1988} 
Saharon Shelah and Juris Stepr\={a}ns, 
{\it PFA implies all automorphisms are trivial}, 
Proceedings of the American Mathematical Society,
104:4 (1988), 1220--1225.
}

\bibitem{Simon 1985}
Petr Simon, 
{\it Applications of independent linked families}, 
Topology, theory and applications (Eger, 1983),
Colloq. Math. Soc. J{\'a}nos Bolyai, vol.~41, 
North-Holland, Amsterdam, 1985, 561--580. 

\bibitem{Sirota 1969}
\hide{
С.\,М.~Сирота, 
{\it Произведение топологических групп и 
экстремальная несвязность}, 
Матем. сб., 79(121):2(6) (1969), 179--192; 
}
S.\,M.~Sirota, 
{\it The product of topological groups 
and extremal disconnectedness}, 
Math.~USSR-Sb., 8:2 (1969), 169--180. 

\hide{
\bibitem{Steen Seebach} 
L.\,A.~Steen, J.\,A.~Seebach (Jr.),  
{\it Counterexamples in Topology}. 
Dover Publications reprint of 1978~ed., 
Berlin, New York, Springer-Verlag, 1995.
}




\hide{
\bibitem{Velickovich 1993} 
Boban Veli\v{c}kovi\'{c}, 
{\it OCA and automorphisms of 
$P(\omega)/\mathrm{fin}$}, 
Topology Appl. 49:1 (1993), 1--13.
}





\bibitem{Zwaneveld 2021} 
Gabri\"{e}lle Zwaneveld, 
{\it Een ruit van ultrafilters}, 
BSc.~thesis, TU Delft, 2021 (Dutch).

\end{thebibliography}
\end{document}